\documentclass[12pt, reqno]{amsart}
\usepackage{amssymb}
\usepackage{amsmath}
\usepackage{versions}
\usepackage{amsfonts}
\usepackage{graphicx}
\usepackage[usenames,dvipsnames] {color}

\usepackage{systeme}

\def\be{\begin{equation}}
\def\ee{\end{equation}}

\def\s0{s_0}
\def\p0{p_0}
\let\phi\varphi
\let\epsilon\varepsilon

\newtheorem{defn}{Definition}[section]
\newtheorem{thm}[defn]{Theorem}
\newtheorem{prop}[defn]{Proposition}
\newtheorem{cor}[defn]{Corollary}
\newtheorem{lem}[defn]{Lemma}

\newtheorem{rem}[defn]{Remark}

\theoremstyle{definition}
\newtheorem{eg}[defn]{Example}
\numberwithin{equation}{section}

\DeclareMathOperator\hol{Hol}

\newcommand{\T}{\mathbb{T}}

\newcommand{\D}{\mathbb{D}}
\newcommand{\C}{\mathbb{C}}

\newcommand{\mat}{\mathbb{C}^{2\times 2}}

\newcommand{\tr}{\operatorname{tr}}

\newcommand\la{\lambda}

\newcommand\spa{\mathrm {span~}}
\newcommand\id{\rm id}

\newcommand\beq{\begin{equation}}

\newcommand\eeq{\end{equation}}

\newcommand\bbm{\begin{bmatrix}}
\newcommand\ebm{\end{bmatrix}}
\newcommand\bpm{\begin{pmatrix}}
\newcommand\epm{\end{pmatrix}}

\let\phi\varphi

\numberwithin{equation}{section}

\begin{document}
\title[A Schwarz lemma for the pentablock]{A Schwarz lemma for the pentablock}
\author{Nujood M. Alshehri and Zinaida A. Lykova}
\date{submitted on 23rd June 2022, accepted on  29th September 2022}

\begin{abstract} In this paper we prove a Schwarz lemma for the pentablock.
The pentablock $\mathcal{P}$ is defined by
\[
\mathcal{P}=\{(a_{21}, \tr A, \det A) : A=[a_{ij}]_{i,j=1}^2 \in \mathbb{B}^{2\times 2}\}
\]
where $\mathbb{B}^{2\times 2}$ denotes the open unit ball in the space of $2\times 2$ complex matrices. 
The pentablock is a bounded nonconvex domain in $\Bbb{C}^3$ which arises naturally in connection with a certain problem of $\mu$-synthesis.
 We  develop a concrete  structure theory for the rational maps from the unit disc $\D$  to the closed pentablock $\overline{\mathcal{P}}$ that map the
unit circle  $\T$ to the distinguished boundary $b\overline{\mathcal{P}}$ of $\overline{\mathcal{P}}$. Such maps are called rational 
${\overline{\mathcal{P}}}$-inner functions.
 We give relations between ${\overline{\mathcal{P}}}$-inner functions and inner functions from $\Bbb{D}$ to the symmetrized bidisc.
We describe the construction of rational ${\overline{\mathcal{P}}}$-inner functions $x = (a, s, p) : \Bbb{D} \rightarrow \overline{\mathcal{P}}$ of prescribed degree from the zeroes of $a, s$ and $s^2-4p$. 
 The proof of this theorem is constructive: it gives an algorithm for the construction of a family of such functions $x$ subject to the computation of  Fej\'er-Riesz factorizations of certain non-negative trigonometric functions on the circle. 
 We use properties and the construction of rational ${\overline{\mathcal{P}}}$-inner functions to prove
  a Schwarz lemma for the pentablock.
\end{abstract}





\maketitle
\tableofcontents

\section{Introduction} \label{intro}

An unsolved problem in $H^\infty$ control theory led us to consider inner rational mappings from $\Bbb{D}$ to certain domains in $\Bbb{C}^d$ which  arise  in connection with the $\mu$-synthesis problem.
One such domain is the pentablock. Other well known examples of such domains are the symmetrized bidisc $\Gamma$ and the tetrablock. We should  mention papers on the construction of rational $\Gamma$-inner functions \cite{ALY18,ALY17} and rational tetra-inner functions \cite{AlsLyk, HAlsLyk} for the symmetrized bidisc $\Gamma$ and the tetrablock respectively.
The pentablock $\mathcal{P}$ was introduced by Agler, Lykova and Young in \cite{ALY2015} in 2015. It was shown there that $\mathcal{P}$ arises naturally in the context of  $\mu$-synthesis. 
\begin{defn}\label{Penta}
\textup{\cite{ALY2015}} The open pentablock is the domain defined by
\begin{equation}\label{pentablock}
\mathcal{P}=\{(a_{21}, \tr A, \det A) : A=[a_{ij}]_{i,j=1}^2 \in \mathbb{B}^{2\times 2}\}
\end{equation}
where $\mathbb{B}^{2\times 2}$ denotes the open unit ball in the space of $\;\;2\times 2$ complex matrices with respect to the operator norm arising from the standard inner product on $\Bbb{C}^2$.
\end{defn} 
Recall \cite{Do} that the {\em structured singular value} $\mu_E$ of $A\in\C^{m\times n}$ corresponding to subspace $E$ of $\C^{n\times m}$ is defined by
\beq\label{defmu}
\frac{1}{\mu_E(A)} = \inf \{\|X\|: X\in E \mbox{ and } \det (1-AX)=0\}.
\eeq
The cost function $\mu_E$ plays a central role in the ``$H^\infty$ approach" to the problem of stabilising a linear system in a way that is maximally robust with respect to structured uncertainty.
This approach, developed and promoted by J. Doyle and G. Stein \cite{Do}, reduces the ``robust stabilization problem" to the solution of a variant of the classical Nevanlinna-Pick problem for matrix-valued functions, in which the cost function to be minimised is given by $\mu_E$ for some uncertainty space $E$, in place of the usual operator norm.

 To date there is not a satisfactory mathematical treatment of this ``$\mu$-synthesis problem" in general, and so mathematicians have studied some special cases, such as for $2\times 2$-matrix-valued functions and for some natural choices of the space $E$. In particular the authors of \cite{ALY2015} investigated the following special case of $\mu_E$.
\begin{defn}\label{defCcc}
Let 
\[
E=\spa\left\{1,\bbm 0&1\\0&0 \ebm\right\} \subset \Bbb{C}^{2\times2},
\]
$\mathcal{P}_\mu$ is the domain in $ \Bbb{C}^3$  given by
\beq\label{defDomain}
\mathcal{P}_\mu= \{(a_{21},\tr A,\det A): A\in\mat, \,  \mu_E(A)<1\} \subset \Bbb{C}^3.
\eeq
\end{defn}
\noindent It was proved in \cite{ALY2015}  that  $\mathcal{P}=\mathcal{P}_\mu$.

The pentablock $\mathcal{P}$ is a region in 3-dimensional complex space which intersects $\Bbb{R}^3$ in a convex body bounded by five faces, comprising two triangles, an ellipse and two curved surfaces \cite{ALY2015}.
The closure of $\mathcal{P}$ is denoted by $\overline{\mathcal{P}}$.

In this paper we study rational $\overline{\mathcal{P}}$-inner functions. We define a rational $\overline{\mathcal{P}}$-inner function to be a rational analytic function from $\Bbb{D}$ into $\overline{\mathcal{P}}$ which maps $\Bbb{T}$ into $b\overline{\mathcal{P}}$, where $b\overline{\mathcal{P}}$ is the distinguished boundary of $\mathcal{P}$. The distinguished boundary $b\overline{\mathcal{P}}$ of $\mathcal{P}$ is
$$b\overline{\mathcal{P}} = \bigg\{(a, s, p) \in \Bbb{C}^3 : |s| \leq 2, \ |p|=1, \ s = \overline{s}p \; \text{and} \; |a| = \sqrt{1-\frac{1}{4}|s|^2}\bigg\},$$
see \cite{ALY2015}.
The degree of a rational $\overline{\mathcal{P}}$-inner function $x = (a, s, p)$ is defined to be the pair of numbers $(\textup{deg} \; a, \textup{deg}\;p)$. We say that deg $x \leq (m, n)$ if deg $a \leq m$ and deg $p \leq n$.
The group of automorphisms of the pentablock was studied in \cite{ALY2015} and \cite{Kos2015}. 

Recall that a classical {\em rational inner function} is a rational map $f$ from the unit disc $\D$  to its closure $\overline{\D}$ with the property that $f$ maps the unit circle $\T$ into itself.  A survey of results connecting inner functions and operator theory is given in \cite{CGP2015}.  All rational inner functions from the unit disc $\D$  to its closure $\overline{\D}$ are finite Blaschke products.
 \begin{defn}\label{finiteBlaschkeprod}
\textup{\cite[page 2]{AgMa}} A \textup{finite Blaschke product} is a function of the form
\begin{equation}\label{fin-Bl-Prod}
B(z) = c\prod^n_{i=1}B_{\alpha_i}(z)  \ \ \ \ \text{for} \ z\in \Bbb{C}\setminus\{1/\overline{\alpha_1},\dots,1/\overline{\alpha_n}\},
\end{equation}
where $B_{\alpha_i}(z) = \frac{z-\alpha_i}{1-\overline{\alpha_i}z}$, 
$|c|=1$ and $\alpha_1, \dots, \alpha_n \in \Bbb{D}$. 
\end{defn}

We have proved several results on the description and the construction of rational $\overline{\mathcal{P}}$-inner functions and on the connections between rational $\Gamma$-inner functions and rational 
$\overline{\mathcal{P}}$-inner functions.

One of our main results is the construction of a rational $\overline{\mathcal{P}}$-inner function $x = (a, s, p)$ of prescribed degree from the zeros of $a$ and $s$ and  $s^2 -4p$.
The zeros of $s^2 -4p$  in $\overline{\Bbb{D}}$ are called the royal nodes of $(s, p)$.
 One can consider this result as an analogue of the expression \eqref{fin-Bl-Prod} for a finite Blaschke product in terms of its zeros. Concretely, the following result is a corollary of Theorem \ref{constructpentainnfunc}.

\begin{thm}\label{constr-intro} Let $n$, $m$ be  positive integers and suppose the following points are given

{\rm (1)}  $\alpha_1, \alpha_2, \dots, \alpha_{k_0} \in \Bbb{D}$ and $\eta_1, \eta_2, \dots, \eta_{k_1} \in \Bbb{T}$, where $2k_0+k_1=n$;

{\rm (2)} $\beta_1, \beta_2, \dots, \beta_m \in \Bbb{D}$;

{\rm (3)}  $\sigma_1, \dots, \sigma_n$ in $\overline{\Bbb{D}}$ which are distinct from $\eta_1, \dots, \eta_{k_1}$.

\noindent Then there exists a rational $\overline{\mathcal{P}}$-inner function $x = (a, s, p)$ of degree $\leq (m+n, n)$ such that
 the zeros of $a$ in $\Bbb{D}$ are $\beta_1, \beta_2, \dots, \beta_m$, the zeros of $s$ in $\overline{\Bbb{D}}$ are $\alpha_1, \alpha_2, \dots, \alpha_{k_0}$, $\eta_1, \eta_2, \dots, \eta_{k_1}$, and the royal nodes of $(s, p)$ are $\sigma_1, \dots, \sigma_n$.
\end{thm}

There is a well developed theory of Schwarz lemmas for various domains by many authors, including Dineen  and Harris \cite{Din,Harris}. In particular, for the symmetrized bidisc and the tetrablock, see \cite{AY2001,AWY,EZ2009}.
Connections established in this paper between  ${\overline{\mathcal{P}}}$-inner functions and $\Gamma$-inner functions (especially Theorem \ref{prop446})  and  a  Schwarz lemma for the symmetrized bidisc due to 
Agler and Young \cite{AY2001}  allow us  to prove  a Schwarz lemma for the pentablock (Theorem \ref{Schwarz-P-inner}).

\begin{thm}\label{Schwarz-P-inner-intro}
Let $\lambda_0 \in \Bbb{D}\setminus\{0\}$ and $(a_0, s_0, p_0) \in \overline{\mathcal{P}}$.
Then the following conditions are equivalent:

{\rm (i)}  there exists a rational $\overline{\mathcal{P}}$-inner function $x=(a, s, p)$,  $x : \Bbb{D} \rightarrow \overline{\mathcal{P}}$ such that \newline $x(0) = (0, 0, 0)$ and $x(\lambda_0) = (a_0, s_0, p_0)$;

{\rm (ii)}  there exists an analytic function $x=(a, s, p)$,  $x : \Bbb{D} \rightarrow \overline{\mathcal{P}}$ such that \newline $x(0) = (0, 0, 0)$ and $x(\lambda_0) = (a_0, s_0, p_0)$, and  $|a_0| \; \le \; |\lambda_0| \sqrt{1-\frac{1}{4}|s_0|^2}$;

{\rm (iii)}
\begin{equation}\label{s0p0h0-int}
\dfrac{2|s_0 - p_0\overline{s}_0| + |s_{0}^{2} - 4p_0|}{4 - |s_0|^2} \le |\lambda_0|\;\; \text{and} \;\;|s_0| <2,
\end{equation}
and
\begin{equation}\label{a0s0h0-int}
|a_0| \; \le \; |\lambda_0| \sqrt{1-\frac{1}{4}|s_0|^2}.
\end{equation}
\end{thm}

The construction of an interpolating function  $x=(a, s, p)$,  $x : \Bbb{D} \rightarrow \overline{\mathcal{P}}$ such that $x(0) = (0, 0, 0)$ and $x(\lambda_0) = (a_0, s_0, p_0)$ is given in Theorem \ref{thmtr2911} and in Theorem \ref{Schwarz-P-inner}.

The authors are grateful to Nicholas Young for some helpful suggestions.  \\


\section{The pentablock $\mathcal{P}$ and the symmetrized bidisc $\Gamma$ }\label{3sec1}

In 1999 Agler and Young introduced the symmetrized bidisc in \cite{AY99}. Following \cite{AY99}, we shall often use the co-ordinates $(s,p)$ for points in the symmetrized bidisc $\Bbb{G}$, chosen to suggest 'sum' and 'product'. 

\begin{defn}\label{symmbidiscdfn}
The symmetrized bidisc is the set
\begin{equation}\label{symm}
\Bbb{G} \stackrel{\mathrm{def}}{=} \{(z+w,zw):|z| \; \textless \; 1, |w| \; \textless \; 1\},
\end{equation}
and its closure is
\begin{equation*}
\Gamma \stackrel{\mathrm{def}}{=} \{(z+w,zw):|z| \leq 1, |w| \leq 1\}.
\end{equation*}
\end{defn}
The following results from \cite{ALY13} give useful criteria for membership of $\Bbb{G}$, of the distinguished boundary $b\Gamma$ of $\Gamma$ and of the topological boundary $\partial \Gamma$ of $\Gamma$.

\begin{prop}\label{symm-bidisc}\textup{\cite[Proposition 3.2]{ALY13}}
Let $(s,p)$ belong to $\Bbb{C}^2$. Then
\begin{enumerate}
\item $(s,p)$ belongs to $\Bbb{G}$ if and only if $$|s- \overline{s}p| \; \textless \; 1-|p|^2;$$
\item $(s,p)$ belongs to $\Gamma$ if and only if $$|s| \leq 2 \textit{ and } |s- \overline{s}p| \leq 1-|p|^2;$$
\item $(s,p)$ lies in $b\Gamma$ if and only if $$|p|=1, |s| \leq 2 \textit{ and } s- \overline{s}p = 0;$$
\end{enumerate}
\end{prop}

The following terminology was introduced in \cite{AY}.

\begin{defn}\label{royalvar}
The royal variety $\mathcal{R}_{\Gamma}$ of the symmetrized bidisc is
$$\mathcal{R}_{\Gamma} = \{(s, p) \in \Bbb{C}^2 : s^2 = 4p\}.$$
\end{defn}
\begin{lem}\label{AutGonroyalvar}
\textup{\cite[Lemma 4.3]{AY2008}} Every automorphism of $\Bbb{G}$ maps the royal variety $\mathcal{R}_{\Gamma} \cap \Bbb{G}$ onto itself.
\end{lem}
The royal variety is the only complex geodesic in the symmetrized bidisc that is invariant under all automorphisms of $\Bbb{G}$ \cite{ALY19}.

\begin{rem}\label{PGrelated} {\rm
The pentablock is closely related to the symmetrized bidisc. Indeed, Definition \ref{Penta} shows that $\mathcal{P}$ is fibred over $\Bbb{G}$ by the map $(a, s, p) \mapsto (s, p)$, since if $A \in \Bbb{B}^{2\times 2}$ then the eigenvalues of $A$ lie in $\Bbb{D}$ and so $(\tr A, \det A) \in \Bbb{G}$. Thus, for every point $(a, s, p) \in \mathcal{P}$, the  point  $(s,p)  \in \Bbb{G}$.}
\end{rem}

\begin{rem} {\rm
In \cite{Kos2015} Kosinski commented that the pentablock is a Hartogs Domain. It follows from the descriptions of 
  the pentablock $\mathcal{P}$ in \cite{ALY2015} that  $\mathcal{P}$ can be seen as a Hartogs domain in $\Bbb{C}^3$ over the symmetrized bidisc $\Bbb{G}$, that is,
$$\mathcal{P} = \Big\{(a, s, p) \in \Bbb{D} \times \Bbb{G} : |a|^2 \; \textless \; e^{-\varphi(s, p)}\Big\},$$
where
$$\varphi(s, p) = -2\log \left|1-\frac{\frac{1}{2}s\overline{\beta}}{1+\sqrt{1-|\beta|^2}}\right|,$$
$(s, p) \in \Bbb{G}$ and $\beta = \frac{s-\overline{s}p}{1-|p|^2}$. \\
Hartogs domains are important objects in several complex variables.
}
\end{rem}

\begin{defn}\label{C-convex}
\textup{\cite[page 259]{JP2013}} A domain $D \subset \Bbb{C}^n$ is called $\Bbb{C}$-convex if for any complex line $\ell = a+b\Bbb{C}, \; 0 \neq a, b \in \Bbb{C}^n$ such that $\ell \cap D \neq \emptyset$, this intersection $\ell \cap D$ is connected and simply connected.
\end{defn}

It is known that the pentablock is polynomially convex and starlike, see \cite{ALY2015}. 
It was shown in \cite{Za2015} that  the pentablock $\mathcal{P}$ is hyperconvex and that $\mathcal{P}$ cannot be exhausted by domains biholomorphic to convex ones. Later in  \cite[Theorem 1.1]{S2020} it was proved that $\mathcal{P}$ is a $\Bbb{C}$-convex domain.\\

The following results from \cite{ALY2015} give useful criteria for membership of $\mathcal{P}$.

\begin{defn}\label{Psi_z(a, s, p),kappa(s, p)}
\textup{\cite[Definition 4.1]{ALY2015}} For $z \in \Bbb{D}$ and $(a, s, p) \in \Bbb{C}^3$ define $\Psi_z(a, s, p)$ by
\begin{equation}\label{Psi_z}
\Psi_z(a, s, p) = \frac{a(1-|z|^2)}{1-sz+pz^2} \ \ \text{ whenever } \ 1-sz+pz^2 \neq 0.
\end{equation}
\end{defn}

The polynomial map implicit in the definition \eqref{pentablock} can be written as
\begin{equation}\label{pi}
\pi(A)=(a_{21}, \tr A, \det A) \; \text{for} \; A=[a_{ij}]_{i,j=1}^2 \in \Bbb{C}^{2\times 2}.
\end{equation}
Thus $\mathcal{P} = \pi(\mathbb{B}^{2\times 2})$. 

\begin{thm}\label{pentathm1.1}
\textup{\cite[Theorem 1.1]{ALY2015}}  Let $$(s,p)=(\lambda_1+\lambda_2,\lambda_1\lambda_2)$$ where $\lambda_1$, $\lambda_2$ $\in$ $\Bbb{D}$. Let a $\in \Bbb{C}$ and let $$\beta=\frac{s-\overline{s}p}{1-|p|^2}.$$
Then $|\beta| \; \textless \; 1$ and the following statements are equivalent:
\begin{enumerate}
\item (a,s,p) $\in \mathcal{P}$, that is, there exists $A \in \Bbb{C}^{2\times 2}$ such that $\|A\| < 1$ and $\pi(A) = (a, s, p)$;
\item $|a| \; \textless \; |1-\frac{\frac{1}{2}s \overline{\beta}}{1+\sqrt{1-|\beta|^2}}|$;
\item $|a| \; \textless \; \frac{1}{2}|1-\overline{\lambda_2}\lambda_1|+\frac{1}{2}(1-|\lambda_1|^2)^\frac{1}{2}(1-|\lambda_2|^2)^\frac{1}{2}$;
\item $\sup_{z\in \Bbb{D}}|\Psi_{z} (a,s,p)| \; \textless \; 1$.
\end{enumerate}
\end{thm}

\begin{thm}\label{pentathm5.3}
\textup{\cite[Theorem 5.3]{ALY2015}} Let
$$(s, p) = (\beta+\overline{\beta}p, p) = (\lambda_1+\lambda_2, \lambda_1\lambda_2)\in \Gamma$$
where $|\beta| \leq 1$ and if $|p| = 1$ then $\beta = \frac{1}{2}s$. Let $a \in \Bbb{C}$. The following statements are equivalent:
\begin{enumerate}
\item $(a, s, p) \in \overline{\mathcal{P}}$;

\item $|a| \leq |1-\frac{\frac{1}{2}s\overline{\beta}}{1+\sqrt{1-|\beta|^2}}|$;
\item $|a| \leq \frac{1}{2}|1-\overline{\lambda}_2\lambda_1| + \frac{1}{2}(1-|\lambda_1|^2)^{\frac{1}{2}}(1-|\lambda_2|^2)^{\frac{1}{2}}$;
\item $|\Psi_z(a, s, p)| \leq 1$ for all $z \in \Bbb{D}$, where $\Psi_z$ is defined by equation \eqref{Psi_z};
\item there exists $A \in \Bbb{C}^{2\times 2}$ such that $\|A\| \leq 1$ and $\pi(A) = (a, s, p).$
\end{enumerate}
\end{thm}

\section{The distinguished boundary of $\mathcal{P}$}\label{distingboundpenta}

Let $\Omega$ be a domain in $\Bbb{C}^n$ with closure $\overline{\Omega}$ and let $A(\Omega)$ be the algebra of continuous scalar functions on $\overline{\Omega}$ that are holomorphic on $\Omega$. A \textit{boundary} for $\Omega$ is a subset $C$ of $\overline{\Omega}$ such that every function in $A(\Omega)$ attains its maximum modulus on $C$. 
Since $\overline{\mathcal{P}}$ is polynomially convex, there is a smallest closed boundary of $\mathcal{P}$, contained in all the closed boundaries of $\mathcal{P}$, called the distinguished boundary of $\mathcal{P}$ and denoted by $b\overline{\mathcal{P}}$. If there is a function $g \in A(\mathcal{P})$ and a point $u \in \overline{\mathcal{P}}$ such that $g(u)=1$ and $|{g(x)}|$ $\textless$ 1 for all $x \in \overline{\mathcal{P}} \backslash \{u\}$, then $u$ must belong to $b\overline{\mathcal{P}}$. Such a point $u$ is called a \textit{peak point} of $\overline{\mathcal{P}}$ and the function $g$ a \textit{peaking function} for $u$.

\vspace{5mm}
Define $$K_0 \stackrel{\textup{def}}{=}\bigg\{(a, s, p) \in \Bbb{C}^3 : (s, p) \in b\Gamma, \ |a| = \sqrt{1-\frac{1}{4}|s|^2}\bigg\}.$$
and
\begin{equation}\label{K1}
K_1 \stackrel{\textup{def}}{=}\bigg\{(a, s, p) \in \Bbb{C}^3 : (s, p) \in b\Gamma, \ |a| \leq \sqrt{1-\frac{1}{4}|s|^2}\bigg\}.
\end{equation}

\begin{prop}\label{pentaprop8.3}
\textup{\cite[Proposition 8.3]{ALY2015}} The subsets $K_0$ and $K_1$ of \ $\overline{\mathcal{P}}$ are closed boundaries for $A(\mathcal{P})$.
\end{prop}

\begin{thm}\label{pentathm8.4}
\textup{\cite[Theorem 8.4]{ALY2015}} For $x \in \Bbb{C}^3$, the following are equivalent: \\
\textup{(1)} $x \in K_0$; \\
\textup{(2)} $x$ is a peak point of $\overline{\mathcal{P}}$; \\
\textup{(3)} $x \in b\overline{\mathcal{P}}$, the distinguished boundary of $\mathcal{P}$. \\
\newline Therefore
$$b\overline{\mathcal{P}} = \bigg\{(a, s, p) \in \Bbb{C}^3 : (s, p) \in b\Gamma, \ |a| = \sqrt{1-\frac{1}{4}|s|^2}\bigg\}$$
and so
\begin{equation}\label{bP}
b\overline{\mathcal{P}} = \bigg\{(a, s, p) \in \Bbb{C}^3 : |s| \leq 2, \ |p|=1, \ s = \overline{s}p \ and \ |a| = \sqrt{1-\frac{1}{4}|s|^2}\bigg\}.
\end{equation}
\end{thm}

\begin{thm}\label{pentathm8.5}
\textup{\cite[Theorem 8.5]{ALY2015}} The distinguished boundary $b\overline{\mathcal{P}}$ is homeomorphic to
$$\{(\sqrt{1-x^2}w, x, \theta) : -1 \leq x \leq 1, \ 0 \leq \theta \leq 2\pi , \ w \in \Bbb{T}\}$$
with the two points $(\sqrt{1-x^2}w, x, 0)$ and $(\sqrt{1-x^2}w, -x, 2\pi)$ identified for every $w \in \Bbb{T}$ and $x \in [-1, 1]$.
\end{thm}

\section{The royal variety of $\mathcal{P}$ and Aut $\mathcal{P}$}\label{singsetpenta}

Recall that $\mathcal{P}=\pi(\Bbb{B}^{2\times2})$ where $\pi:\Bbb{C}^{2\times2}\rightarrow \Bbb{C}^3$ is defined as 
\begin{equation}\label{piA}
\pi:A \mapsto (a_{21}, \textup{tr }A, \textup{det }A).
\end{equation}
We define the singular set of $\mathcal{P}=\pi(\Bbb{B}^{2\times2})$ to be the image under $\pi$ of the set of critical points of $\pi$.

\begin{prop}\label{singsetpentaprop}
The singular set of the pentablock is $\mathcal{R}_{\mathcal{P}}=\{(0, s, p )\in \mathcal{P}: s^2=4p\}$.
\end{prop}

\begin{proof} The set of critical points of $\pi$ is the set
$\pi(\{A \in \Bbb{B}^{2\times2}: \mathbf{J}_{\pi}(A)$ is not of full rank\}), where $\mathbf{J}_{\pi}(A)$ is the Jacobian matrix of $\pi$. \\

The Jacobian matrix of $\pi$, $\mathbf{J}_{\pi}(A)$, is defined by

\begin{equation*} 
\mathbf{J}_{\pi}(A)=
\begin{bmatrix}
\dfrac{\partial \pi_1}{\partial a_{11}} & \dfrac{\partial \pi_1}{\partial a_{12}} & \dfrac{\partial \pi_1}{\partial a_{21}} & \dfrac{\partial \pi_1}{\partial a_{22}} \\ \\ \dfrac{\partial \pi_2}{\partial a_{11}} & \dfrac{\partial \pi_2}{\partial a_{12}} & \dfrac{\partial \pi_2}{\partial a_{21}} & \dfrac{\partial \pi_2}{\partial a_{22}} \\ \\ \dfrac{\partial \pi_3}{\partial a_{11}} &
\dfrac{\partial \pi_3}{\partial a_{12}} & \dfrac{\partial \pi_3}{\partial a_{21}} & \dfrac{\partial \pi_3}{\partial a_{22}} &
\end{bmatrix}\;\; \text{for} \;\; A=
\begin{bmatrix}
a_{11} & a_{12} \\ a_{21} & a_{22}
\end{bmatrix} \in \Bbb{C}^{2 \times 2}.
\end{equation*}
Thus,
\begin{equation*}
\mathbf{J}_\pi(A)=
\begin{bmatrix}
0 & 0 & 1 & 0 \\ 1 & 0 & 0 & 1 \\ a_{22} & -a_{21} & -a_{12} & a_{11}
\end{bmatrix}.
\end{equation*}
Note that $\mathbf{J}_{\pi}(A)$ is not of full rank if and only if rank $\mathbf{J}_{\pi}(A) \leq 2$. That means all $3\times 3$ minors of $\mathbf{J}_{\pi}(A)$ are zero. Let us find all $3\times 3$ minors of $\mathbf{J}_\pi(A)$. 
\begin{eqnarray*}
\begin{vmatrix}
 0 & 0 & 1\\     1 & 0 & 0 \\ a_{22} & -a_{21} & -a_{12} 
\end{vmatrix}=-a_{21},
 & ~~ &
\begin{vmatrix}
 0 & 1 & 0\\0 & 0 & 1 \\ -a_{21} & -a_{12} & a_{11} 
\end{vmatrix}=-a_{21},
\end{eqnarray*}
\begin{eqnarray*}
\begin{vmatrix}
 0 & 0 & 0\\1 & 0 & 1 \\ a_{22} & -a_{21} & a_{11} 
\end{vmatrix}=0,
 & ~~ &
\begin{vmatrix}
0 & 1 & 0\\1 & 0 & 1 \\ a_{22} & -a_{12} & a_{11} 
\end{vmatrix}=-a_{11}+a_{22}.
\end{eqnarray*}
Thus $\mathbf{J}_\pi(A)$ is not of full rank if and only if $a_{21}=0$ and $a_{11}=a_{22}$.
Therefore, \begin{eqnarray*}
\mathcal{R}_{\mathcal{P}} = \pi(\{A \in \Bbb{B}^{2\times 2}:\mathbf{J}_\pi(A) \textrm{ is not of full rank} \}) & = & \pi \Bigg( A = \begin{bmatrix} a & * \\ 0 & a \end{bmatrix} \in \Bbb{B}^{2\times 2} \Bigg)\\
& = & \bigg\{ \Big(0, 2a, a^2 \Big): a \in \Bbb{D}\bigg\}\\
& = & \bigg\{ \Big(0, s, p \Big): (s, p)\in \Bbb{G}, s^2=4p\bigg\}.
\end{eqnarray*}
Here $s=$ tr $A$ and $p=$ det $A$.
\end{proof}

\begin{rem}
The singular set of the pentablock can be presented as $$\mathcal{R}_{\mathcal{P}} = \{(0, s, p) \in \mathcal{P} : (s, p) \in \mathcal{R}_\Gamma \cap  \mathcal{G}\}.$$ 
\end{rem}

\noindent By analogy with the established terminology for the symmetrized bidisc, we shall call the set 
$$ \mathcal{R}_{\overline{\mathcal{P}}}= \{(0, s, p) \in \C^3 : s^2= 4p\}.$$
the {\em royal variety} of the pentablock.

\begin{lem}\label{a=0iff}
Let $(a, s, p) \in b\overline{\mathcal{P}}$. Then the following conditions are equivalent:\\
 {\rm (i)}  $a = 0$;\\
 {\rm (ii)}  $(a, s, p) \in b\overline{\mathcal{P}} \cap \mathcal{R}_{\overline{\mathcal{P}}}$;\\
 {\rm (iii)}  $|s|=2$.
\end{lem}

\begin{proof}  It easily follows from the definition of 
$\mathcal{R}_{\overline{\mathcal{P}}}$ and the formula \eqref{bP} for $b\overline{\mathcal{P}}$.
\end{proof}

\vspace{5mm}
{\bf The automorphism group of $\mathcal{P}$}.
Recall the known information on the automorphism group Aut $\mathcal{P}$ of $\mathcal{P}$ from \cite{ALY2015}.
For $w \in \Bbb{T}$ and $v \in \textup{Aut} \ \Bbb{D}$, let
\begin{equation}\label{fwv}
f_{wv}(a, s, p) = \Big(\frac{w\eta(1-|\alpha|^2)a}{1-\overline{\alpha}s+\overline{\alpha}^2p}, \tau_v(s, p) \Big)
\end{equation}
where $v=\eta B_\alpha$ for $\alpha \in \Bbb{D}$, $\eta \in \Bbb{T}$, $B_\alpha(z) = \dfrac{z-\alpha}{1-\overline{\alpha}z}$ is a Blaschke factor and $\tau_v \in \textup{Aut} \ \Bbb{G}$ is defined by
$$\tau_v(z+w, zw) = \big(v(z)+v(w), v(z)v(w)\big).$$
\index{$f_{wv}(a, s, p)$}

\begin{thm}\label{pentathm7.1}
\textup{\cite[Theorem 7.1]{ALY2015}} The maps $f_{wv}$, for $w \in \Bbb{T}$ and $v \in \textup{Aut} \ \Bbb{D}$, constitute a group of automorphisms of $\mathcal{P}$ under composition. Each automorphism $f_{wv}$ extends analytically to a neighbourhood of $\overline{\mathcal{P}}$. \\
Moreover, for all $w_1, w_2 \in \Bbb{T}$, $v_1, v_2 \in \textup{Aut} \ \Bbb{D}$,
$$f_{w_1v_1} \circ f_{w_2v_2} = f_{(w_1w_2)(v_1\circ v_2)},$$
and, for all $w \in \Bbb{T}$, $v \in \textup{Aut} \ \Bbb{D}$,
$$(f_{wv})^{-1} = f_{\overline{w}v^{-1}}.$$ 
\end{thm}

\vspace{5mm}
L. Kosi\'nski proved in \cite{Kos2015} that the set $\{f_{wv}:w \in \Bbb{T}, v \in \textup{Aut} \ \Bbb{D}\}$ is the full group of automorphisms of $\mathcal{P}$.

\vspace{5mm}
\begin{lem}\label{singsetandAutpenta}
$\mathcal{R}_{\overline{\mathcal{P}}} \cap \mathcal{P}$ is invariant under \textup{Aut} $\mathcal{P}$.
\end{lem}

\begin{proof}
Every element of $\mathcal{R}_{\overline{\mathcal{P}}} \cap \mathcal{P}$ is of the form $(0, s, p) \in \mathcal{P}$ where $s^2 = 4p$.  It is easy to see that, for every  
element $f_{wv}$ of Aut $\mathcal{P}$ given by the equation \eqref{fwv}, 
$$f_{wv}(0, s, p) = \Big(0, \tau_v(s, p) \Big).$$
Since $\tau_v \in \textup{Aut} \ \Bbb{G}$ and $(s, p)\in \mathcal{R}_\Gamma \cap \mathbb{G}$, by Lemma \ref{AutGonroyalvar}, $\tau_v(s, p) \in \mathcal{R}_\Gamma \cap \mathbb{G}$.
Therefore, $f_{wv}(0, s, p) \in \mathcal{R}_{\overline{\mathcal{P}}}.$
\end{proof}

\vspace*{0.3cm}

For any domain  $U$ in $\Bbb{C}^{n}$, $ \text{Hol}(\Bbb{D}, U)$ denotes the space of analytic functions from $\Bbb{D}$ to $U$.

\begin{defn}\label{Complexgeodesic}
Let $U$ be a domain in $\Bbb{C}^n$ and let $\mathcal{D} \subset U$. We say $\mathcal{D}$ is a complex geodesic in $U$ if there exists a function $k \in \textup{Hol}(\Bbb{D}, U)$ and a function $C \in \textup{Hol}(U, \Bbb{D})$ such that $C \circ k = \id_{\Bbb{D}}$ and $\mathcal{D} = k(\Bbb{D})$.
\end{defn}

For a geometric classification of complex geodesics in the symmetrized bidisc $ \Bbb{G}$, see  \cite{ALYMEM}.
We define, for $ \omega \in \T $, the rational function $\Phi_\omega$ of two variables by
$$ \Phi_\omega(s, p) = \frac{2\omega p-s}{2-\omega s}, \; \; \text{where} \;\; \omega s \neq 2 .$$
In the function theory and geometry of $ \Bbb{G}$ much depends on the properties of these functions $\Phi_\omega,\; \omega \in \T$, see \cite{AY2008}. 

\begin{lem}\label{Rpcompgeo}
$\mathcal{R}_{\overline{\mathcal{P}}} \cap \mathcal{P}$ is a complex geodesic in $\mathcal{P}$.
\end{lem}

\begin{proof}
Define the analytic functions $k$ and $c$ by
$$k:\Bbb{D}\rightarrow \mathcal{P}, \; k(\lambda) = (0, -2\lambda, \lambda^2)$$
and
$$c:\mathcal{P} \rightarrow \Bbb{D}, \; c(a, s, p) = \Phi_\omega(s, p) = \frac{2\omega p-s}{2-\omega s}, \; \; \text{where} \;\; \omega \in \T .$$
For $\lambda \in \Bbb{D}$,
$$(c \; \circ \; k)(\lambda) = c(k(\lambda)) = c(0, -2\lambda, \lambda^2) = \frac{2\omega\lambda^2+2\lambda}{2+2\omega\lambda}  = \lambda,$$  which means $ c \; \circ \; k = \id_{\Bbb{D}}.$
By definition of $\mathcal{R}_{\overline{\mathcal{P}}}$, it is easy to see that $\mathcal{R}_{\overline{\mathcal{P}}}  \cap \mathcal{P} = k(\Bbb{D})$. Therefore $\mathcal{R}_{\overline{\mathcal{P}}} \cap \mathcal{P}$ is a complex geodesic in $\mathcal{P}$.
\end{proof}

\section{Examples of $\overline{\mathcal{P}}$-inner functions}\label{exofpentainner}

Descriptions of inner and outer functions in $H^\infty(\Bbb{D})$ and  properties of inner and outer functions can be found  \cite[Chapter III]{SF1970}. Here $H^\infty(\Bbb{D})$ is the space of holomorphic functions $u$ on $\Bbb{D}$ such that the corresponding norm
\begin{equation*}
\|u\|_\infty = \sup\limits_{\lambda \in \Bbb{D}} |u(\lambda)|
\end{equation*}
is finite.

\begin{defn}\label{innerfuncDtoD}
An \textup{inner function} is an analytic map $f:\Bbb{D} \rightarrow \overline{\Bbb{D}}$ such that the radial limit $$\lim_{r \to 1^{-}}f(r\lambda)$$ exists and belongs to $\Bbb{T}$ for almost all $\lambda \in \Bbb{T}$ with respect to Lebesgue measure.
\end{defn}

It is well-known that the rational inner functions on $\Bbb{D}$ are precisely the finite Blaschke products. 
One can see that the only functions which are at the same time inner and outer are the constant functions of modulus 1.

\begin{defn}\label{pentainnerfunc}
A $\overline{\mathcal{P}}$-inner or penta-inner function is an analytic map $f:\Bbb{D}\rightarrow \overline{\mathcal{P}}$ such that the radial limit $$\lim_{r \to 1^{-}}f(r\lambda)$$ exists and belongs to $b\overline{\mathcal{P}}$ for almost all $\lambda \in \Bbb{T}$ with respect to Lebesgue measure.
\end{defn}
By Fatou's Theorem, the $\lim_{r \to 1^{-}}f(r\lambda)$ exists for almost all $\lambda \in \Bbb{T}$.

\begin{rem}\label{rem4.2Omar}
Let $f:\Bbb{D}\rightarrow \overline{\mathcal{P}}$ be a rational $\overline{\mathcal{P}}$-inner function. Since $f$ is rational and bounded on $\Bbb{D}$ it has no poles in $\overline{\Bbb{D}}$ and hence $f$ is continuous on $\overline{\Bbb{D}}$. Thus one can consider the continuous function $$\tilde f:\Bbb{T}\rightarrow b\overline{\mathcal{P}}, \textrm{ where }  \tilde f(\lambda)=\lim_{r \to 1^-}f(r\lambda) \textrm{ for all } \lambda \in \Bbb{T}.$$
\end{rem}

\begin{eg}\label{pentainnerfunex1}
Let us consider an example of an analytic function $f:\Bbb{D}\rightarrow \overline{\mathcal{P}}$. Consider the analytic map $h:\Bbb{D}\rightarrow \Bbb{B}^{2\times 2}$ defined by
\begin{equation}\label{hdef} h(\lambda)=
\begin{bmatrix}
\varphi(\lambda) & 0 \\ 0 & \psi(\lambda)
\end{bmatrix} \ \ \ \textrm{for} \ \lambda \in \Bbb{D},
\end{equation}
where $\varphi, \psi \in H^\infty(\Bbb{D})$ are nonconstant inner functions. 
Note that $$\|h(\lambda)\|= \max \{|\varphi(\lambda)|,|\psi(\lambda)|\} \; \textless \; 1\; \text{ for} \; \lambda \in \Bbb{D}.$$
For all $\lambda \in \Bbb{D}$, let $$f(\lambda) = \pi(h(\lambda))=(0, \tr h(\lambda), \det h(\lambda)).$$

\noindent {\em Let us show that the function $f$ is a $\overline{\mathcal{P}}$-inner function only when $\varphi = \psi$.}\\

Let, for $\lambda \in \Bbb{D}, \; a(\lambda) = 0, \; s(\lambda) = \tr h(\lambda) = \varphi(\lambda)+\psi(\lambda)$ and $p(\lambda) = \det h(\lambda) = \varphi(\lambda)\psi(\lambda)$.
Clearly $f=(a, s, p):\Bbb{D}\rightarrow \overline{\mathcal{P}}$ is an analytic function.

To prove that $f$ is $\overline{\mathcal{P}}$-inner, we need to check that $f(\lambda) \in b\overline{\mathcal{P}}$ for almost every $\lambda \in \Bbb{T}$, that is, $\big(s(\lambda), p(\lambda)\big) \in b\Gamma$ and $\sqrt{1-\frac{1}{4}|s(\lambda)|^2} = 0$ for almost every $\lambda \in \Bbb{T}$.
Note that, for almost every $\lambda \in \Bbb{T}$,
$$|p(\lambda)| = |\varphi(\lambda)\psi(\lambda)| = |\varphi(\lambda)||\psi(\lambda)| = 1, \ \textrm{ since } \ |\varphi(\lambda)| = 1 \textrm{ and } \ |\psi(\lambda)| = 1,$$
$$|s(\lambda)| = |\varphi(\lambda)+\psi(\lambda)| \leq |\varphi(\lambda)|+|\psi(\lambda)| = 2,$$ and
\begin{eqnarray*}
(\overline{s}p)(\lambda) & = & \big(\overline{\varphi(\lambda)+\psi(\lambda)}\big)\big(\varphi(\lambda) \psi(\lambda)\big) = \varphi(\lambda)\overline{\varphi(\lambda)}\psi(\lambda) + \varphi(\lambda)\psi(\lambda)\overline{\psi(\lambda)} \\
& = & |\varphi(\lambda)|^2 \psi(\lambda) + \varphi(\lambda)|\psi(\lambda)|^2 = \varphi(\lambda)+\psi(\lambda) = s(\lambda).
\end{eqnarray*}
Hence for almost every $\lambda \in \Bbb{T}, \ |p(\lambda)| = 1, \ |s(\lambda)|\leq 2$ and $(\overline{s}p)(\lambda) = s(\lambda)$, and so $\big(\tr h(\lambda), \det h(\lambda)\big) \in b\Gamma$. 
Now, for almost every $\lambda \in \Bbb{T}$,
\begin{eqnarray*}
1-\frac{1}{4}|s(\lambda)|^2 & = & 1-\frac{1}{4}|\varphi(\lambda)+\psi(\lambda)|^2 = 1-\frac{1}{4}\Big(\big(\varphi(\lambda)+\psi(\lambda)\big)\big(\overline{\varphi(\lambda)+\psi(\lambda)}\big)\Big) \\
& = & 1-\frac{1}{4}\Big(1+1+2 \mathrm{Re}\big(\varphi(\lambda)\overline{\psi(\lambda)}\big)\Big) =
\frac{1}{2}-\frac{1}{2}\mathrm{Im}\big(i\varphi(\lambda)\overline{\psi(\lambda)}\big).
\end{eqnarray*}
Hence $|a|=\sqrt{1-\frac{1}{4}|s|^2}$ almost everywhere on $\Bbb{T}$ if and only if
$$\frac{1}{2}-\frac{1}{2} \mathrm{Im} \big(i\varphi(\lambda)\overline{\psi(\lambda)}\big) = 0, \textrm{  for almost every  }\lambda \in \Bbb{T},$$
if and only if $\mathrm{Im}(i\varphi(\lambda)\overline{\psi(\lambda)}) = 1 \textrm{  for almost every  }\lambda \in \Bbb{T}.$
Therefore, $|a|=\sqrt{1-\frac{1}{4}|s|^2}$ almost everywhere on $\Bbb{T}$ if and only if $\varphi(\lambda)\overline{\psi(\lambda)} = 1$ almost everywhere on $\Bbb{T}$, and so, $\varphi(\lambda) =  \psi(\lambda)$ almost everywhere on $\Bbb{T}$.
Thus the function $f$ is a $\overline{\mathcal{P}}$-inner function only when $\varphi = \psi$, and so
$f = (0, 2 \varphi, \varphi^2)$.
\end{eg}

\begin{eg}\label{pentainnerfunex2}
Let $h_1 : \Bbb{D} \rightarrow \Bbb{C}^{2 \times 2}$ be defined by $h_1=Uh$, where $h$ is defined by equation \textup{(\ref{hdef})} and
\begin{equation*}
U=
\begin{bmatrix}
\dfrac{1}{\sqrt{2}} & \dfrac{1}{\sqrt{2}} \\ 
\dfrac{1}{\sqrt{2}}i & -\dfrac{1}{\sqrt{2}}i
\end{bmatrix}.
\end{equation*}
Note that $U$ is a unitary matrix.
Then, for $\lambda \in \Bbb{D}$,
\begin{eqnarray*}
h_1(\lambda) & = & Uh(\lambda) \\
                    & = & \frac{1}{\sqrt{2}} \begin{bmatrix} \varphi(\lambda) & \psi(\lambda) \\ i\varphi(\lambda) & -i\psi(\lambda) \end{bmatrix},
\end{eqnarray*}
and, for all $\lambda \in \Bbb{D}$,
$$\|h_1(\lambda)\| \leq \|U\| \|h(\lambda)\|=\|h(\lambda)\|=\max\{|\varphi(\lambda)|, |\psi(\lambda)|\} < 1.$$ 
Hence $h_1(\lambda) \in \Bbb{B}^{2\times2}$ for all $\lambda \in \Bbb{D}$. Define $f_1=\pi \circ h_1$ on $\Bbb{D}$. Then, for $\lambda \in \Bbb{D}$,
\begin{align}\label{2ndegf_1}
\notag f_1(\lambda) & = \pi(h_1(\lambda)) \\
\notag           & = \pi \bigg(\frac{1}{\sqrt{2}} \begin{bmatrix} \varphi(\lambda) & \psi(\lambda) \\ i\varphi(\lambda) & -i\psi(\lambda) \end{bmatrix}\bigg) \\
                    & = \bigg(\frac{i\varphi(\lambda)}{\sqrt{2}}, \frac{\varphi(\lambda)-i\psi(\lambda)}{\sqrt{2}}, -i\varphi(\lambda)\psi(\lambda)\bigg).
\end{align}
Clearly, $f_1:\Bbb{D}\rightarrow \overline{\mathcal{P}}$ is an analytic function since $\varphi, \psi$ are analytic on $\Bbb{D}$.\\

\noindent{\em Let us show that $f_1$  is a $\overline{\mathcal{P}}$-inner function if and only if $\varphi = i\psi$.}\\

Let us check when the function $f_1$ is $\overline{\mathcal{P}}$-inner. We need to find conditions when $f_1$ maps $\Bbb{T}$ into the distinguished boundary $b\overline{\mathcal{P}}$ of $\mathcal{P}$.
Since $\varphi$, $\psi$ are inner functions, they have unit modulus almost everywhere on $\Bbb{T}$. Thus, for $s = \dfrac{\varphi-i\psi}{\sqrt{2}}$, $p = -i\varphi \psi$ and for almost every $\lambda \in \Bbb{T}$,
$$|p(\lambda)|=|-i\varphi(\lambda)\psi(\lambda)|=|\varphi(\lambda)||\psi(\lambda)|=1.$$
$$|s(\lambda)|=\bigg|\frac{\varphi(\lambda)-i\psi(\lambda)}{\sqrt{2}}\bigg|
 \leq \frac{|\varphi(\lambda)|+|-i\psi(\lambda)|}{\sqrt{2}}=\frac{2}{\sqrt{2}}.$$
\begin{eqnarray*}
(\overline{s}p)(\lambda) & = & \frac{\overline{\varphi(\lambda)-i\psi(\lambda)}}{\sqrt{2}}(-i\varphi(\lambda)\psi(\lambda)) \\                                           
 & = & \frac{-i|\varphi(\lambda)|^2 \psi(\lambda)+\varphi(\lambda)|\psi(\lambda)|^2}{\sqrt{2}} =
                       \frac{\varphi(\lambda)-i\psi(\lambda)}{\sqrt{2}}=s(\lambda).
\end{eqnarray*}
Therefore, for almost every $\lambda \in \Bbb{T}, |p(\lambda)| = 1$, $|s(\lambda)|\leq 2$ and $(\overline{s}p)(\lambda) = s(\lambda)$ and so $(s(\lambda), p(\lambda)) \in b\Gamma$. Finally,
\begin{eqnarray*}
\sqrt{1-\frac{1}{4}|s(\lambda)|^2} & = & \sqrt{1-\frac{1}{4}  \bigg(\frac{1}{2}|\varphi(\lambda)-i \psi(\lambda)|^2\bigg)} \\
  & = & \sqrt{1-\frac{1}{8}\bigg(1+1+2 \mathrm{Re} (\varphi(\lambda) i \overline{\psi(\lambda)})\bigg)} 
                                      = \frac{1}{2} \sqrt{3+\mathrm{Im} (\varphi(\lambda)\overline{\psi(\lambda)})}.
\end{eqnarray*}
We want $|a|=\sqrt{1-\frac{1}{4}|s|^2}$ almost everywhere on $\Bbb{T}$, that is, for almost every $\lambda \in \Bbb{T}$,
$$\dfrac{1}{\sqrt{2}} = \dfrac{|\varphi(\lambda)|}{\sqrt{2}} = \dfrac{1}{2} \sqrt{3+\mathrm{Im} (\varphi(\lambda)\overline{\psi(\lambda)})}.$$
Hence $|a|=\sqrt{1-\frac{1}{4}|s|^2}$ almost everywhere on $\Bbb{T}$ if and only if $\sqrt{3+\mathrm{Im} (\varphi(\lambda)\overline{\psi(\lambda)})} = \sqrt{2}$ for almost every $\lambda \in  \Bbb{T}$, or equivalently,
 $\varphi(\lambda) = -i\psi(\lambda)$ for almost every $\lambda \in \Bbb{T}$. Thus $f_1$ given by equation \eqref{2ndegf_1} is a $\overline{\mathcal{P}}$-inner function if and only if $\varphi = -i\psi$.  In this case
 $f_1=  \bigg(\frac{i\varphi}{\sqrt{2}},  \sqrt{2} \varphi  , \varphi^2\bigg)$. 
\end{eg}

\begin{eg}\label{pentainnerfunex3}
Let  $v$, $\varphi$ and $\psi$ be inner functions on $\Bbb{D}$.
Consider the functions
\begin{equation*}
V(\lambda) = \frac{1}{\sqrt{2}}\begin{bmatrix} 1 & v \\ -1 & v \end{bmatrix}(\lambda) \ \textrm{ and } \ h(\lambda) = \begin{bmatrix} \varphi & 0 \\ 0 & \psi \end{bmatrix}(\lambda),\; \textrm{ for }\;  \lambda \in \Bbb{D}.
\end{equation*}
Define
\begin{eqnarray*}
U(\lambda) & = & (V^* h V)(\lambda)=
\frac{1}{\sqrt{2}} \begin{bmatrix} 1 & -1 \\ \overline{v} & \overline{v} \end{bmatrix}
\begin{bmatrix} \varphi & 0 \\ 0 & \psi \end{bmatrix}
\frac{1}{\sqrt{2}} \begin{bmatrix} 1 & v \\ -1 & v \end{bmatrix}(\lambda) \\
& = & \frac{1}{2} \begin{bmatrix} \varphi + \psi & (\varphi - \psi)v \\ (\varphi - \psi)\overline{v} & \varphi + \psi \end{bmatrix}(\lambda), \ \textrm{ for }  \lambda \in \Bbb{D}.
\end{eqnarray*}
Note that, for all  $ \lambda \in \Bbb{D}$,
$$\| V^*(\lambda)\| = \left\|
 \begin{bmatrix} 1 & 0 \\ 0 & \overline{v}(\lambda) \end{bmatrix} \frac{1}{\sqrt{2}}
\begin{bmatrix} 1 & -1 \\ 1 & 1 \end{bmatrix} \right\| \leq 1,$$
since 
$  \frac{1}{\sqrt{2}}\begin{bmatrix} 1 & -1 \\ 1 & 1 \end{bmatrix}$ is unitary and 
$| \overline{v}(\lambda) | \le 1$ on $\Bbb{D}$.
Hence
$$\|U(\lambda)\| \leq\|V(\lambda)\|^2 \|h(\lambda)\| < 1 \ \textrm{ for }  \lambda \in \Bbb{D}.$$
Define $f:\Bbb{D} \rightarrow \overline{\mathcal{P}}$  by  $f=\pi \circ U$.
Then, for $\lambda \in \Bbb{D}$,
\begin{equation}\label{3rdegf}
f(\lambda) = \pi \circ U(\lambda) = \Big(\frac{1}{2}(\varphi - \psi)\overline{v}, \varphi + \psi, \frac{1}{4}\big((\varphi + \psi)^2-(\varphi - \psi)^2|v|^2\big) \Big)(\lambda).
\end{equation}
Note that $f$ is analytic on $\Bbb{D}$ if and only if $v$ is constant or $\varphi = \psi$. \\

\noindent{\em Let us show that $f$  is a $\overline{\mathcal{P}}$-inner function if and only if  $v$ is constant or $\varphi = \psi$.}\\

{\bf Case 1}. Suppose $v$ is constant. As $v$ is inner, $|v|=1$.
Let us check that $f:\Bbb{D} \rightarrow \overline{\mathcal{P}}$ is $\overline{\mathcal{P}}$-inner, that is, $f(\Bbb{T}) \subset b\overline{\mathcal{P}}$. Note, for almost all $\lambda \in \Bbb{T}$, 
$\det U(\lambda) = (\varphi \psi)(\lambda)$ and $(\varphi+\psi, \varphi\psi)(\lambda) \in b\Gamma$ as in Example \ref{pentainnerfunex1}.
$$|a|^2 = \dfrac{1}{4}|\varphi - \psi|^2 = \dfrac{1}{4} \Big(1+1-2 \mathrm{Re} (\overline{\varphi}\psi) \Big) = \dfrac{1}{2}-\dfrac{1}{2} \mathrm{Re} (\overline{\varphi}\psi) \textrm{ almost everywhere on } \Bbb{T}.$$
$$1 - \frac{1}{4} |s|^2 = 1 - \frac{1}{4}|\varphi + \psi|^2 = 1 - \frac{1}{4} \Big(1 + 1 + 2 \mathrm{Re}(\overline{\varphi}\psi) \Big) = 1-\frac{1}{2} \mathrm{Re} (\overline{\varphi}\psi) = |a|^2.$$
Thus $|a|^2 = 1-\frac{1}{4}|s|^2$ almost everywhere on $\Bbb{T}$, and so, $f$ given by equation \eqref{3rdegf} is a $\overline{\mathcal{P}}$-inner function if $v$ is constant.  Since $|v|=1$,
$f =  \Big(\frac{1}{2}(\varphi - \psi)\overline{v}, \varphi + \psi, \varphi  \psi \Big)$. \\

{\bf Case 2}. Suppose  $\varphi = \psi$. Then 
$$f = \Big(0, 2\varphi, \frac{1}{4}(2\varphi)^2\Big) = (0, 2\varphi, \varphi^2).$$
We have shown in Example \ref{pentainnerfunex1} that $f= (0, 2\varphi, \varphi^2)$ is a $\overline{\mathcal{P}}$-inner function.
\end{eg}

\begin{eg}\label{pentainnerfunex4}
Define the function $x(\lambda) = (\lambda^m, 0, \lambda): \Bbb{D}\rightarrow \overline{\mathcal{P}}$. First we need to show that for all $\lambda \in \Bbb{D}, \; x(\lambda) \in \overline{\mathcal{P}}$.  {\em Let us show that $x$ is a rational $\overline{\mathcal{P}}$-inner function.}\\

By Proposition \ref{symm-bidisc}, $(s, p) \in \Gamma$ if and only if
$$|s|\leq 2 \ \textrm{  and  } |s-\overline{s}p|\leq 1-|p|^2.$$
It is easy to see that $(0, \lambda) \in \Gamma$.
By Theorem \ref{pentathm5.3}, if $(s, p) \in \Gamma$, then for $a \in \Bbb{C}, (a, s, p) \in \overline{\mathcal{P}}$ if and only if
\begin{equation}\label{*}
|a| \leq \left|1-\frac{\frac{1}{2}s\overline{\beta}}{1+\sqrt{1-|\beta|^2}}\right|, \ \text{ where } \ \beta=\frac{s-\overline{s}p}{1-|p|^2}.
\end{equation}
In the case $s = 0$, equation (\ref{*}) is equivalent to $|a| \leq 1$. 
Note that $x(\lambda) = \big(a(\lambda), s(\lambda), p(\lambda)\big) = (\lambda^m, 0, \lambda), \; \lambda \in \Bbb{D}, \text{ is analytic in } \Bbb{D}$, and
$$|a(\lambda)| = |\lambda^m| \leq 1, \ \textrm{ for all } \ \lambda \in \Bbb{D}.$$
Thus for $\lambda \in \Bbb{D}, \ x(\lambda) \in \overline{\mathcal{P}}$. \\

Now, let us check that $x$ maps $\Bbb{T}$ into the distinguished boundary $b\overline{\mathcal{P}}$ of $\mathcal{P}$. For all $\lambda \in \Bbb{T}$,
$$|p(\lambda)| = |\lambda| = 1, \ |s(\lambda)| = |0| \leq 2,$$
$$(\overline{s}p)(\lambda) = 0 = s(\lambda) \ \textrm{ and }$$
$$|a(\lambda)| = |\lambda^m| = \sqrt{1-\frac{1}{4}|s(\lambda)|^2} = 1.$$
Therefore for every $\lambda \in \Bbb{T}$, $x(\lambda) \in b\overline{\mathcal{P}}$ and hence $x$ is a rational $\overline{\mathcal{P}}$-inner function.
\end{eg}

\begin{eg}\label{pentainnerfunex5}
For $\lambda \in \Bbb{D}$, define the function $x(\lambda) = \big(a(\lambda), s(\lambda), p(\lambda)\big)= (\lambda, 0, \lambda^n)$. 
As in the previous example, for all $\lambda \in \Bbb{D}$, by Proposition \ref{symm-bidisc},
$(0, \lambda^n) \in \Gamma$. For $a \in \Bbb{C}$, we want
\begin{equation}\label{**}
|a| \leq |1-\frac{\frac{1}{2}s\overline{\beta}}{1+\sqrt{1-|\beta|^2}}|.
\end{equation}
Since $s = 0$, the condition (\ref{**}) is equivalent to $|a| \leq 1$. Note that
$$|a(\lambda)| = |\lambda| \leq 1, \textrm{ for all } \ \lambda \in \Bbb{D}.$$
Thus, by Theorem \ref{pentathm5.3}, for $\lambda \in \Bbb{D}, \ x(\lambda) \in \overline{\mathcal{P}}$. {\em Let us show that $x$ is a rational $\overline{\mathcal{P}}$-inner function.}\\

Let us check that $x$ maps $\Bbb{T}$ into the distinguished boundary $b\overline{\mathcal{P}}$ of $\mathcal{P}$. For all $\lambda \in \Bbb{T}$,
$$|p(\lambda)| = |\lambda^n| = 1, \ |s(\lambda)| = |0| \leq 2,$$
$$(\overline{s}p)(\lambda) = 0 = s(\lambda) \ \ \textrm{and }$$
$$|a(\lambda)| = |\lambda| = \sqrt{1-\frac{1}{4}|s(\lambda)|^2} = 1.$$
Therefore for every $\lambda \in \Bbb{T}$, $x(\lambda) \in b\overline{\mathcal{P}}$ and hence $x$ is a rational $\overline{\mathcal{P}}$-inner function.
\end{eg}

\section{Some properties of analytic functions $x:\Bbb{D}\rightarrow \overline{\mathcal{P}}$}\label{propanalyfuncs}

In this section and in Section \ref{Connectiongammapenta}
 we will show  that there are close relations between 
$\overline{\mathcal{P}}$-inner functions and  $\Gamma$-inner functions.
 Recall that the $\Gamma$-inner functions were first mentioned in \cite{2x2NPP}. A good undersdanding of  rational $\Gamma$-inner functions will play an important part in any future solution of the finite interpolation problel for
$\text{Hol}(\Bbb{D}, \Gamma)$, since such a problem has a solution if and only if it has a rational $\Gamma$-inner solution  (see, for example, \cite[Theorem 4]{costara2005} and  \cite[Theorem 8.1]{ALY13-2}).
The rational $\Gamma$-inner functions were classified in \cite{ALY13}. Algebraic and geometric aspects of rational $\Gamma$-inner functions were presented in \cite{ALY18}. 

\begin{defn}\label{Gammainnerfunc}
A $\Gamma$-inner function is an \textit{analytic function} $h:\Bbb{D} \rightarrow \Gamma$ such that the radial limit \begin{equation} \label{radiallimit}
\lim_{r \to 1^-} h(r \lambda)
\end{equation}
exists and belongs to $b\Gamma$ for almost all $\lambda \in \Bbb{T}$ with respect to Lebesgue measure.
\end{defn}

\begin{lem}\label{pentagammainner}
{\rm (i)}  Let $x = (a, s, p) : \Bbb{D} \rightarrow \overline{\mathcal{P}}$ be an analytic function. Then $h = (s, p) : \Bbb{D} \rightarrow \Gamma$ is an analytic function.

{\rm (ii)}
Let $x = (a, s, p) : \Bbb{D} \rightarrow \overline{\mathcal{P}}$ be a $\overline{\mathcal{P}}$-inner function. Then $h = (s, p) : \Bbb{D} \rightarrow \Gamma$ is a $\Gamma$-inner function.
\end{lem}

\begin{proof}
(i) By assumption, $x = (a, s, p)$ is analytic on $\Bbb{D}$ and for all $\lambda \in \Bbb{D}$, \\
$x(\lambda) = (a(\lambda), s(\lambda), p(\lambda)) \in \overline{\mathcal{P}}$. By Remark \ref{PGrelated}, for all $\lambda \in \Bbb{D}$, $(s(\lambda), p(\lambda)) \in \Gamma$. Thus $h = (s, p):\Bbb{D} \rightarrow \Gamma$, where $h(\lambda) = (s(\lambda), p(\lambda))$, for $\lambda \in \Bbb{D}$, is well-defined and analytic from $\Bbb{D}$ to $\Gamma$. 

(ii) By assumption $x = (a, s, p):\Bbb{D} \rightarrow \overline{\mathcal{P}}$ is a penta-inner function, and so, for almost all $\lambda \in \Bbb{T}$, $x(\lambda) \in b\overline{\mathcal{P}}$. Recall $b\overline{\mathcal{P}} = \Big\{(a, s, p) \in \Bbb{C}^3 : (s, p) \in b\Gamma, \; |a| = \sqrt{1-\frac{1}{4}|s|^2}\Big\}$. By Theorem \ref{pentathm8.4}, for almost all $\lambda \in \Bbb{T}$, $h(\lambda) = \big(s(\lambda), p(\lambda)\big) \in b\Gamma$. Hence $h$ is a $\Gamma$-inner function. 
\end{proof}

\vspace*{0.5cm}
Recall that, by Proposition \ref{pentaprop8.3}, $K_1 = \bigg\{(a, s, p) \in \overline{\mathcal{P}} : (s, p) \in b\Gamma, \ |a| \leq \sqrt{1-\frac{1}{4}|s|^2}\bigg\}$ is a closed boundary of $A(\mathcal{P})$.

\begin{prop}\label{Gamma-inner-penta}
\textup{(i)} Let $h = (s, p) : \Bbb{D} \rightarrow \Gamma$ be an analytic function. Then $x = (0, s, p)$ is an analytic function from $\Bbb{D}$ to $\overline{\mathcal{P}}$. \newline
\textup{(ii)} Let $h = (s, p) : \Bbb{D} \rightarrow \Gamma$ be a $\Gamma$-inner function. Then $x = (0, s, p):\Bbb{D} \rightarrow \overline{\mathcal{P}}$ is an analytic function such that, for almost all $\lambda \in \Bbb{T}, \; x(\lambda) \in K_1$.
\end{prop}

\begin{proof}
(i) It follows from Theorem \ref{pentathm5.3} that, for all $\lambda \in \Bbb{D}, \big(0, s(\lambda), p(\lambda)\big) \in \overline{\mathcal{P}}$. \\
(ii) Suppose that $h$ is a $\Gamma$-inner function. By Proposition \ref{symm-bidisc}, $|p(\lambda)| = 1, \ |s(\lambda)| \leq 2 \ \text{ and } \ (\overline{s}p)(\lambda) = s(\lambda)$, for almost all $\lambda \in \Bbb{T}$.
Since $a = 0 \text{ and } \sqrt{1-\frac{1}{4}|s(\lambda)|^2} \geq 0$ for almost all $\lambda \in \Bbb{T}$, $x(\Bbb{T}) \subset K_1$.
\end{proof}

\begin{prop}\label{(a_out,s,p)2}
Let $x = (a, s, p)$ be a $\overline{\mathcal{P}}$-inner function. Let $a_{in} \; a_{out}$ be the inner-outer factorization of $a$. Then $\widetilde{x} = (a_{out}, s, p)$ is a $\overline{\mathcal{P}}$-inner function.
\end{prop}

\begin{proof}
By assumption $x = (a, s, p)$ is a $\overline{\mathcal{P}}$-inner function. Hence, for each $\lambda \in \Bbb{D}$, $\;(a(\lambda), s(\lambda), p(\lambda)) \in \overline{\mathcal{P}}$. By Theorem \ref{pentathm5.3}, for all $\lambda \in \Bbb{D}$, $|\Psi_z(a(\lambda), s(\lambda), p(\lambda))| \leq 1$ for all $z \in \Bbb{D}.$ Thus
$$\left|\frac{a(\lambda)(1-|z|^2)}{1-s(\lambda)z+p(\lambda)z^2}\right| \leq 1, \; \text{ for all } \lambda, z \in \Bbb{D}.$$
Recall that  $a=a_{in}a_{out}$, where $a_{in}$ is inner, and so 
 $|a_{in}(\lambda)| = 1$ for almost all $\lambda \in \Bbb{T}$. 
 Therefore, for every $z \in \Bbb{D}$, and, for almost all $\lambda \in \Bbb{T}$,
 $$\left|a_{in}(\lambda)\frac{a_{out}(\lambda)(1-|z|^2)}{1-s(\lambda)z+p(\lambda)z^2}\right|=
 \left|\frac{a_{out}(\lambda)(1-|z|^2)}{1-s(\lambda)z+p(\lambda)z^2}\right| \leq 1. $$
 Note that, for every $z \in \Bbb{D}$, the function 
 $$ \lambda \mapsto  \frac{a_{out}(\lambda)(1-|z|^2)}{1-s(\lambda)z+p(\lambda)z^2}$$ is analytic on $\Bbb{D}.$
 By the maximum principle for analytic functions,  for every $z \in \Bbb{D}$,
$$\left|\frac{a_{out}(\lambda)(1-|z|^2)}{1-s(\lambda)z+p(\lambda)z^2}\right| \leq 1, \; \text{ for all } \lambda \in \Bbb{D}.$$
Hence, by Theorem \ref{pentathm5.3}, for each $\lambda \in \Bbb{D}, \; (a_{out}(\lambda), s(\lambda), p(\lambda)) \in \overline{\mathcal{P}}$. Therefore,  $\widetilde{x} = (a_{out}, s, p) \in$ Hol$(\Bbb{D}, \overline{\mathcal{P}})$. 

To prove the statement of the proposition, we must show that, for almost all $\lambda \in \Bbb{T}$, $\widetilde{x}(\lambda) = (a_{out}(\lambda), s(\lambda), p(\lambda)) \in b\overline{\mathcal{P}}$. Recall that \[b\overline{\mathcal{P}} = \left\{(a, s, p) \in \Bbb{C}^3 : (s, p) \in b\Gamma, \; |a|=\sqrt{1-\frac{1}{4}|s|^2}\right\}.\] By Lemma \ref{pentagammainner}, for almost all $\lambda \in \Bbb{T}$, $(s(\lambda), p(\lambda)) \in b\Gamma$. Since $x = (a, s, p)$ is a $\overline{\mathcal{P}}$-inner function, we have, for almost all $\lambda \in \Bbb{T}$,
$$|a(\lambda)| = \sqrt{1-\frac{1}{4}|s(\lambda)|^2}.$$
Since $a_{in}a_{out} = a$ is the inner-outer factorization of $a$ and $|a_{in}(\lambda)| = 1$ for almost all $\lambda \in \Bbb{T}$,
$$|a_{out}(\lambda)| = \sqrt{1-\frac{1}{4}|s(\lambda)|^2} \; \text{ for almost all } \lambda \in \Bbb{T}.$$
Therefore $\widetilde{x} = (a_{out}, s, p)$ is a $\overline{\mathcal{P}}$-inner function.
\end{proof}

\section{Connections between rational $\Gamma$-inner and rational $\overline{\mathcal{P}}$-inner functions}\label{Connectiongammapenta}

\begin{thm}\label{Fejerthm}
\textbf{\textup{(Fej\'er-Riesz theorem)}} \textup{\cite[Section 53]{FN1990}} If $f(\lambda) = \sum^{n}_{i=-n}a_i\lambda^i$ is a trigonometric polynomial of degree $n$ such that $f(\lambda) \geq 0$ for all $\lambda \in \Bbb{T}$, then there exists an analytic polynomial $D(\lambda) = \sum^n_{i=0}b_i\lambda^i$ of degree $n$ such that $D$ is outer (that is, $D(\lambda)\neq 0$ for all $\lambda \in \Bbb{D}$) and
$$f(\lambda)=|D(\lambda)|^2$$
for all $\lambda \in \Bbb{T}$.
\end{thm}
\index{Fej\'er-Riesz theorem}

Recall that for every $a \neq 0$ in $H^\infty(\Bbb{D})$ there is an outer-inner factorization. Rational inner functions can be written in the form $c\prod\limits_{i=1}^{n}B_{\alpha_{i}}$ for some $n \geq 1$ and $\alpha_1, \dots, \alpha_n \in \Bbb{D}$ and $c \in \Bbb{C}$.

\begin{defn}\label{Gammadefn3.1}
\textup{\cite[Definition 3.1]{ALY18}} The degree \textup{deg}$(h)$ of a rational $\Gamma$-inner function $h$ is defined to be $h_*(1)$, where $h_* : \Bbb{Z}=\pi_1(\Bbb{T}) \rightarrow \pi_1(b\Gamma)$ is the homomorphism of fundamental groups induced by $h$ when it is regarded as a continuous map from $\Bbb{T}$ to $b\Gamma$.
\end{defn}

Recall that, by \cite[Proposition 3.3]{ALY18},  for any rational $\Gamma$-inner function $h=(s, p)$, \textup{deg}$(h)$ is the degree \textup{deg}$(p)$ (in the usual sense) of the finite Blaschke product $p$.

\begin{defn}\label{polyfsimn}
Let $g$ be a polynomial of degree less than or equal to $n$, where $n \geq 0$. Then we define the polynomial $g^{\sim n}$ by $$g^{\sim n}(\lambda)= \lambda^n \overline{g(1/\overline{\lambda})}.$$
\end{defn}

\begin{defn}\label{degreeofrapeninn}
The degree of a rational $\overline{\mathcal{P}}$-inner function $x = (a, s, p)$ is defined to be the pair of numbers $(\textup{deg}\;a, \textup{deg}\;p)$. We say that \textup{deg} $x \leq (m, n)$ if \textup{deg} $a \leq m$ and \textup{deg} $p \leq n$.
\end{defn}

The next theorem provides a description of the structure of rational penta-inner functions of prescribed degree.

\begin{thm}\label{descripration}
Let $x=(a, s, p) : \Bbb{D} \rightarrow \overline{\mathcal{P}}$ be a rational penta-inner function of degree $(m, n)$. Let $a \neq 0$ and let an inner-outer factorization of $a$ be given by $a = a_{in} a_{out}$, where $a_{in}$ is an inner function and $a_{out}$ is an outer function. Then there exist polynomials $A, E, D$ such that
\begin{enumerate}
\item $deg(A), deg(E), deg(D) \leq n$,
\item $E^{\sim n} = E$,
\item $D(\lambda) \neq 0$ on $\overline{\Bbb{D}}$,
\item $|E(\lambda)| \leq 2|D(\lambda)|$ on $\overline{\Bbb{D}}$,
\item $A$ is an outer polynomial such that $|A(\lambda)|^2 = |D(\lambda)|^2 - \frac{1}{4}|E(\lambda)|^2$ on $\Bbb{T}$,
\item $a = a_{in} \dfrac{A}{D}$ on $\overline{\Bbb{D}}$,
\item $s = \dfrac{E}{D}$ on $\overline{\Bbb{D}}$,
\item $p = \dfrac{D^{\sim n}}{D}$ on $\overline{\Bbb{D}}$.

\end{enumerate}
\end{thm}

\begin{proof}
Suppose that $x=(a, s, p)$ is a rational penta-inner function. By Lemma \ref{pentagammainner}, $h=(s, p)$ is a rational $\Gamma$-inner function. By \cite[Corollary 6.10]{ALY13}, $p$ can be written in the form
$$p(\lambda) = c\frac{\lambda^kD^{\sim (n-k)}(\lambda)}{D(\lambda)}$$
where $|c|=1$, $0 \leq k \leq n$ and $D$ is a polynomial of degree $n-k$ such that $D(0)=1$.
Therefore, by \cite[Proposition 2.2]{ALY18}, there exist polynomials $E$ and $D$ such that
\begin{eqnarray}\label{E-D-s-p}
\text{(i)}   & \text{deg}(E), \text{deg}(D) \leq n, \nonumber\\
\text{(ii)}  & E^{\sim n}=E,\nonumber \\
\text{(iii)} & D(\lambda) \neq 0 \textit{ on } \overline{\Bbb{D}},\nonumber\\
\text{(iv)}  & |E(\lambda)| \leq 2 |D(\lambda)| \textit{ on } \overline{\Bbb{D}},\\
\text{(v)}   & s=\dfrac{E}{D} \textit{ on } \overline{\Bbb{D}},\nonumber\\
\text{(vi)}  & p = \dfrac{D^{\sim n}}{D} \text{ on } \overline{\Bbb{D}}.\nonumber
\end{eqnarray}
By assumption $x = (a, s, p)$ is a $\overline{\mathcal{P}}$-inner function, and so, for almost all $\lambda \in \Bbb{T}, \; (a(\lambda), s(\lambda), p(\lambda)) \in b\overline{\mathcal{P}}$, which implies
$$|a_{out}(\lambda)|^2 = 1-\frac{1}{4}|s(\lambda)|^2, \; \text{ since } |a_{in}(\lambda)| = 1 \text{ almost everywhere on } \Bbb{T}.$$
Thus
$$|a_{out}(\lambda)|^2 = 1-\frac{1}{4}\frac{|E(\lambda)|^2}{|D(\lambda)|^2} \ \ \text{ since } s(\lambda) = \frac{E(\lambda)}{D(\lambda)},$$
and so,
\begin{equation}\label{equation1}
|a_{out}(\lambda)|^2|D(\lambda)|^2 = |D(\lambda)|^2-\frac{1}{4}|E(\lambda)|^2.
\end{equation}
By \cite[Proposition 2.2]{ALY18}, $|E(\lambda)| \leq 2|D(\lambda)|$.
By the Fej\'er-Riesz Theorem, since $|D(\lambda)|^2-\frac{1}{4}|E(\lambda)|^2 \geq 0$, there exists an analytic polynomial $A$ of degree $\leq n$ such that $A$ is outer and
\begin{equation}\label{equation2}
|A(\lambda)|^2=|D(\lambda)|^2-\frac{1}{4}|E(\lambda)|^2
\end{equation}
for all $\lambda \in \Bbb{T}$. \\
From equations \eqref{equation1} and \eqref{equation2} we have,
$|A(\lambda)|^2 = |a_{out}(\lambda)|^2 |D(\lambda)|^2$. Note that $D(\lambda)\neq 0$ on $\overline{\Bbb{D}}$. Thus $|a_{out}(\lambda)| = \left|\dfrac{A}{D}(\lambda)\right|$ for $\lambda \in \Bbb{T}$, and so $\dfrac{A}{D}$ is an outer function such that $|a(\lambda)| = \left|\dfrac{A}{D}(\lambda)\right|$ for almost all $\lambda \in \Bbb{T}$. Since outer factors are unique up to unimodular constant multiples, there exists $\omega \in \Bbb{T}$ such that
$$a_{out}(\lambda) = \omega\dfrac{A(\lambda)}{D(\lambda)}.$$
Therefore $a = a_{in} \dfrac{A}{D}$ on $\overline{\Bbb{D}}$, after replacement of $A$ by $ \omega A$. 
\end{proof}

\begin{rem} 
Results similar to our Theorem \ref{descripration} were  announced on ArXiv in \cite{AJPK}.
\end{rem}

\begin{eg}\label{D,E,A,ai}  Consider  a rational $\overline{\mathcal{P}}$-inner function  $x(\lambda) = (\lambda^m, 0, \lambda)$ for $\lambda \in \Bbb{D}$. It is easy to see that polynomials described in Theorem \ref{descripration}
for this function are the following: $E(\lambda) = 0$, $D(\lambda) = 1$, $D^{\sim 1}(\lambda) = \lambda$, $A(\lambda) = 1$. Since $a(\lambda) = \lambda^m$ is inner, $a_{in}= a$ and so  $a_{in}(\lambda) = \lambda^m$. 
\end{eg}

\begin{thm}\label{prop446}
Let $h = (s, p) : \Bbb{D} \rightarrow \Gamma$ be a rational $\Gamma$-inner function of degree $n$. Let $E, D$ be defined by equations \eqref{E-D-s-p}
{\rm (\cite[Proposition 2.2]{ALY18})}. Let $A$ be an outer polynomial such that
\begin{equation}\label{|A(lambda)|^2}
|A(\lambda)|^2 = |D(\lambda)|^2-\frac{1}{4}|E(\lambda)|^2.
\end{equation}
Then, for every finite Blaschke product $B$ and $|c|=1$, $x=\left(c B \dfrac{A}{D}, \dfrac{E}{D}, \dfrac{D^{\sim n}}{D}\right)$ is a rational $\overline{\mathcal{P}}$-inner function.
\end{thm}

\begin{proof}
Let $a, s, p$ be defined by
$$a = c B \dfrac{A}{D}, \; \; \; s = \dfrac{E}{D} \; \; \; \text{ and } p = \dfrac{D^{\sim n}}{D}.$$
Let us show that $x = (a, s, p)$ is a rational $\overline{\mathcal{P}}$-inner function. We have to prove that $x : \Bbb{D}\rightarrow \overline{\mathcal{P}}$ and, for almost all $\lambda \in \Bbb{T}, \; x(\lambda) \in b\overline{\mathcal{P}}$. \\
By assumption $h = (s, p) : \Bbb{D} \rightarrow \Gamma$ is a rational $\Gamma$-inner function, which means $|p(\lambda)| = 1, \ |s(\lambda)| \leq 2 \ \text{ and } \ (\overline{s}p)(\lambda) = s(\lambda)$, for almost all $\lambda \in \Bbb{T}$. Now we need to show that for almost all $\lambda \in \Bbb{T}, \; |a(\lambda)| = \sqrt{1-\frac{1}{4}|s(\lambda)|^2}$. For almost all $\lambda \in \Bbb{T}$,
\begingroup
\addtolength{\jot}{1em}
\begin{align*}
|a(\lambda)|^2 & = \left|c B(\lambda) \dfrac{A(\lambda)}{D(\lambda)}\right|^2 = \dfrac{|A(\lambda)|^2}{|D(\lambda)|^2} \; \; \; \text{ (since } |c| = 1 \text{ and } |B(\lambda)| = 1 \text{ on } \Bbb{T}) \\
& = \dfrac{|D(\lambda)|^2-\frac{1}{4}|E(\lambda)|^2}{|D(\lambda)|^2} = 1-\frac{1}{4}\left|\frac{E(\lambda)}{D(\lambda)}\right|^2 \\
& = 1-\frac{1}{4}|s(\lambda)|^2.
\end{align*}
\endgroup

Let us show that $x = (a, s, p) = \left(c B \dfrac{A}{D}, \dfrac{E}{D}, \dfrac{D^{\sim n}}{D}\right)$ maps $\Bbb{D}$ to $\overline{\mathcal{P}}$, that is, $x(\lambda) = (a(\lambda), s(\lambda), p(\lambda)) \in \overline{\mathcal{P}}$ for all $\lambda \in \Bbb{D}$. By the construction, $D(\lambda) \neq 0$ on $\overline{\Bbb{D}}$, and so $(a(\lambda), s(\lambda), p(\lambda))$ is analytic on $\Bbb{D}$. By Theorem \ref{pentathm5.3}, for each $\lambda \in \Bbb{D}, \; x(\lambda) \in \overline{\mathcal{P}}$ if and only if $|\Psi_z(x(\lambda))| \leq 1 \text{ for all } z \in \Bbb{D},$ where
\begin{align*}
\Psi_z(x(.)) : & \; \Bbb{D} \rightarrow \Bbb{C} \\
                & \; \lambda \mapsto (1-|z|^2) \; \frac{a(\lambda)}{1-s(\lambda)z+p(\lambda)z^2}.
\end{align*}
For all $\lambda \in \Bbb{D}$, $(s(\lambda), p(\lambda)) \in \Gamma$, and so $1-s(\lambda)z+p(\lambda)z^2\neq 0$ for all $z \in \Bbb{D}$. Hence, for every $z \in \Bbb{D},\; \Psi_z(x(.))$ is analytic on $\Bbb{D}$. For fixed $z \in \Bbb{D}$, by the maximum principle, to prove that $|\Psi_z(x(\lambda))| \leq 1$ for all $\lambda \in \Bbb{D}$, it suffices to show that $|\Psi_z(x(\lambda))| \leq 1$ for all $\lambda \in \Bbb{T}$. We have shown above that, for almost all $\lambda \in \Bbb{T}, \; (a(\lambda), s(\lambda), p(\lambda)) \in b\overline{\mathcal{P}}$. Thus, for all $\lambda \in \Bbb{T}$, $|a(\lambda)| = \sqrt{1-\frac{1}{4}|s(\lambda)|^2}, \; |p(\lambda)| = 1, \; |s(\lambda)| \leq 2$ and $s(\lambda) = \overline{s(\lambda)}p(\lambda)$, and so $(s(\lambda), p(\lambda)) = (\beta+\overline{\beta}p, p)(\lambda) \in b\Gamma$, where $\beta(\lambda)=\frac{1}{2}s(\lambda)$. One can see that, for all $\lambda \in \Bbb{T}$,
\begin{align*}
\left|1-\dfrac{\frac{1}{2}s(\lambda)\overline{\beta}(\lambda)}{1+\sqrt{1-|\beta(\lambda)|^2}}\right| & = 
\left|1-\frac{\frac{1}{4}|s(\lambda)|^2}{1+\sqrt{1-\frac{1}{4}|s(\lambda)|^2}}\right| \\
& =\left|\frac{1+\sqrt{1-\frac{1}{4}|s(\lambda)|^2}-\frac{1}{4}|s(\lambda)|^2}{1+\sqrt{1-\frac{1}{4}|s(\lambda)|^2}}\right| \\
& = \left|\frac{\sqrt{1-\frac{1}{4}|s(\lambda)|^2} \left( 1 +\sqrt{1-\frac{1}{4}|s(\lambda)|^2} \right)}{1+\sqrt{1-\frac{1}{4}|s(\lambda)|^2}}\right| \\
& = \sqrt{1-\frac{1}{4}|s(\lambda)|^2} = |a(\lambda)|.
\end{align*}
By Theorem \ref{pentathm5.3} (3) $\Leftrightarrow$ (5), for each $\lambda \in \Bbb{T}$,
$$|a(\lambda)| \leq \left|1-\dfrac{\frac{1}{2}s(\lambda)\overline{\beta}(\lambda)}{1+\sqrt{1-|\beta(\lambda)|^2}}\right| \text{ if and only if } |\Psi_z(a(\lambda), s(\lambda), p(\lambda))| \leq 1 \; \text{for all} \; z \in  \Bbb{D}.$$
Hence, by the maximum principle, for all $z, \lambda \in \Bbb{D}$, $|\Psi_z(a(\lambda), s(\lambda), p(\lambda))| \leq 1$. Thus, by Theorem \ref{pentathm5.3}, $x(\lambda) = (a(\lambda), s(\lambda), p(\lambda)) \in \overline{\mathcal{P}}$ for all $\lambda \in \Bbb{D}$.
\end{proof}

\begin{thm}{\bf \textup{(Converse to Theorem \ref{descripration})}}\label{converseofthm443}
Suppose polynomials $A, E, D$ satisfy
\begin{enumerate}
\item $deg(A), deg(E), deg(D) \leq n$,
\item $E^{\sim n} = E$,
\item $D(\lambda) \neq 0$ on $\overline{\Bbb{D}}$,
\item $|E(\lambda)| \leq 2|D(\lambda)|$ on $\overline{\Bbb{D}}$,
\item $A$ is an outer polynomial such that $|A(\lambda)|^2 = |D(\lambda)|^2-\frac{1}{4}|E(\lambda)|^2$ for $\lambda \in \Bbb{T}$,
\item $a_{in}$ is a rational inner function on $\Bbb{D}$ of degree $\leq m$.
\end{enumerate}
Let $a, s, p$ be defined by
$$a = a_{in}\dfrac{A}{D}, \; \; \; s = \dfrac{E}{D} \; \; \; \text{ and } \; \; \; p = \dfrac{D^{\sim n}}{D} \; \; \; \text{ on } \overline{\Bbb{D}}.$$
Then
$$x = \left(a_{in} \dfrac{A}{D}, \dfrac{E}{D},\dfrac{D^{\sim n}}{D}\right)$$
is a rational $\overline{\mathcal{P}}$-inner function of degree less than or equal $(m+n, n)$.
\end{thm}

\begin{proof}
By the converse of \cite[Proposition 2.2]{ALY18}, $h = (s, p)$, where
$$s = \dfrac{E}{D} \; \; \; \text{ and } \; \; \; p = \dfrac{D^{\sim n}}{D},$$
is a rational $\Gamma$-inner function of degree at most $n$. Since the rational inner functions on $\Bbb{D}$ are precisely the finite Blaschke products, the statement of the theorem follows from Theorem \ref{prop446}.\\
\end{proof}

\section{Construction of rational $\overline{\mathcal{P}}$-inner functions}

In this section we describe an algorithm for the construction of rational $\overline{\mathcal{P}}$-inner function from certain interpolation data. Firstly we recall some notions and statements from \cite{ALY18} which were useful for the construction of rational $\Gamma$-inner functions.

\begin{defn}\label{Gammaroyalpoly}
\textup{\cite[Page 140]{ALY18}} Let $h =(s, p)$ be a rational $\Gamma$-inner function of degree $n$. Let $E$ and $D$ be as in 
equations \eqref{E-D-s-p} {\rm (\cite[Proposition 2.2]{ALY18})}. The royal polynomial $R_h$ of $h$ is defined by
\begin{equation}\label{R_h}
R_h(\lambda) = 4D(\lambda)D^{\sim n}(\lambda)-E(\lambda)^2.
\end{equation}
\end{defn}

We call the points $\lambda \in \overline{\Bbb{D}}$ such that $h(\lambda) \in \mathcal{R}_{\Gamma}$ the {\em royal nodes} of $h$ and, for such $\lambda$, we call $h(\lambda)$ a {\em royal point} of $h$, that is, $4p(\lambda)-s(\lambda)^2=0$.
Since $D(\lambda)\neq 0$ on $\overline{\Bbb{D}}$, the royal nodes of $h$ exactly correspond to the zeros of the royal polynomial $R_h$. Hence, $\lambda \in \overline{\Bbb{D}}$ is a royal node of $h$ if and only if $R_h(\lambda) = 0$.

\begin{defn}\label{Gammadefn3.4}
\textup{\cite[Definition 3.4]{ALY18}} We say that a polynomial $f$ is $n$-symmetric if \textup{deg}$(f)\leq n$ and $f^{\sim n} = f$. For any set $E \subset \Bbb{C}$, \textup{ord}$_E(f)$ will denote the number of zeros of $f$ in $E$, counted with multiplicity, and \textup{ord}$_0(f)$ will mean the same as \textup{ord}$_{\{0\}}(f)$.
\end{defn}

\begin{defn}\label{Gammadefn4.1}
\textup{\cite[Definition 4.1]{ALY18}} A nonzero polynomial $R$ is $n$-balanced if \textup{deg}$(R) \leq 2n, \; R$ is $2n$-symmetric and $\lambda^{-n}R(\lambda) \geq 0$ for all $\lambda \in \Bbb{T}$.
\end{defn}

\begin{prop}\label{Gammaprop3.5}
\textup{\cite[Proposition 3.5]{ALY18}} Let $h$ be a rational $\Gamma$-inner function of degree $n$ and let $R_h$ be the royal polynomial of $h$ as defined by equation \eqref{R_h}. Then $R_h$ is $2n$-symmetric and the zeros of $R_h$ that lie on $\Bbb{T}$ have either even or infinite order.
\end{prop}

\begin{defn}\label{Gammadefn3.6}
\textup{\cite[Definition 3.6]{ALY18}} Let $h$ be a rational $\Gamma$-inner function such that $h(\overline{\Bbb{D}})\nsubseteq \mathcal{R}_\Gamma \cap \Gamma$ and let $R_h$ be the royal polynomial of $h$. If $\sigma$ is a zero of $R_h$ of order $\ell$, we define the multiplicity $\#\sigma$ of $\sigma$ (as a royal node of $h$) by
\begin{equation*}
\#\sigma = \begin{cases}
\ell & \text{ if } \sigma \in \Bbb{D} \\
\frac{1}{2}\ell & \text{ if } \sigma \in \Bbb{T}.
\end{cases}
\end{equation*}
\end{defn}

We next present a description of rational penta-inner functions $(a,s,p)$ in terms of the zeros of $a$, $s$ and $s^2 -4p$.

\begin{thm}\label{constructpentainnfunc}
Suppose that $\alpha_1, \alpha_2, \dots, \alpha_{k_0} \in \Bbb{D}$ and $\eta_1, \eta_2, \dots, \eta_{k_1} \in \Bbb{T}$, where $2k_0+k_1=n$ and suppose that $\beta_1, \beta_2, \dots, \beta_m \in \Bbb{D}$. Suppose that $\sigma_1, \dots, \sigma_n$ in $\overline{\Bbb{D}}$ are distinct from $\eta_1, \dots, \eta_{k_1}$. Then there exists a rational $\overline{\mathcal{P}}$-inner function $x = (a, s, p)$ of degree less than or equal $(m+n, n)$ such that
\begin{enumerate}
\item the zeros of $a$ in $\Bbb{D}$, repeated according to multiplicity, are $\beta_1, \beta_2, \dots, \beta_m$,
\item the zeros of $s$ in $\overline{\Bbb{D}}$, repeated according to multiplicity, are $\alpha_1, \alpha_2, \dots, \alpha_{k_0}$ and $\eta_1, \eta_2, \dots, \eta_{k_1}$,
\item the royal nodes of $(s, p)$ are $\sigma_1, \dots, \sigma_n$.
\end{enumerate}
Such a function $x$ can be constructed as follows. Let $t_+ \; \textgreater \; 0$ and let $t \in \Bbb{R} \setminus \{0\}$. Let $R$ and $E$ be defined by
$$R(\lambda) = t_{+} \prod_{j=1}^{n} (\lambda-\sigma_j)(1-\overline{\sigma_j}\lambda),$$
$$E(\lambda) = t\prod_{j=1}^{k_0} (\lambda-\alpha_j)(1-\overline{\alpha_j}\lambda) \prod_{j=1}^{k_1} i\textup{e}^{-i \theta_j/2}(\lambda-\eta_j)$$
where $\eta_j = \textup{e}^{i\theta_j}, \; 0 \leq \theta_j \; \textless \; 2\pi$.
Let $a_{in} : \Bbb{D} \rightarrow \overline{\Bbb{D}}$ be defined by
\begin{equation}\label{a_in}
a_{in}(\lambda) = c\prod_{i=1}^{m} B_{\beta_i}(\lambda),
\end{equation}
where $|c|=1$ and $\beta_i \in \Bbb{D}, \; i=1, \dots, m$.\\

{\rm (i)}  There exist outer polynomials $D$ and $A$ of degree at most $n$ such that
\begin{equation}\label{RED}
\lambda^{-n}R(\lambda)+|E(\lambda)|^2 = 4|D(\lambda)|^2
\end{equation}
and
\begin{equation}\label{RA}
\lambda^{-n}R(\lambda) = 4|A(\lambda)|^2
\end{equation}
for all $\lambda \in \Bbb{T}$.\\

{\rm (ii)} The function $x$ defined by
\begin{equation}\label{x}
x=(a, s, p)=\Big(a_{in}\frac{A}{D}, \frac{E}{D}, \frac{D^{\sim n}}{D}\Big)
\end{equation}
is a rational $\overline{\mathcal{P}}$-inner function such that \textup{deg}$(x) \leq (m+n, n)$ and conditions \textup{(1), (2)} and \textup{(3)} hold. The royal polynomial of $(s, p)$ is $R$.
\end{thm}

\begin{proof}
\textup{(i)} By \cite[Lemma 4.4]{ALY18}, $R$ is $n$-balanced, and so $\lambda^{-n}R(\lambda) \geq 0$ for all $\lambda \in \Bbb{T}$. Therefore
$$\lambda^{-n}R(\lambda)+|E(\lambda)|^2 \geq 0 \text{ for all } \lambda \in \Bbb{T}.$$
By the Fej\'er-Riesz theorem, there exist outer polynomials $A$ and $D$ of degree at most $n$ such that
$$\lambda^{-n}R(\lambda) = 4|A(\lambda)|^2 \; \text{ for all } \lambda \in \Bbb{T}$$
and
$$\lambda^{-n}R(\lambda)+|E(\lambda)|^2 = 4|D(\lambda)|^2 \; \text{ for all } \lambda \in \Bbb{T}.$$
\textup{(ii)} By \cite[Theorem 4.8]{ALY18}, the function $h$ defined by
$$h=(s, p)=\Big(\frac{E}{D}, \frac{D^{\sim n}}{D}\Big)$$
is a rational $\Gamma$-inner function such that \textup{deg}$(h)=n$ and conditions \textup{(2)} and \textup{(3)} hold. The royal polynomial of $h$ is $R$. \\
By equations \eqref{RED} and \eqref{RA},
\begin{eqnarray*}
|A(\lambda)|^2 & = & |D(\lambda)|^2-\frac{1}{4}|E(\lambda)|^2.
\end{eqnarray*}
Therefore, by Proposition \ref{prop446},
$$x = \Big(a_{in}\frac{A}{D}, \frac{E}{D}, \frac{D^{\sim n}}{D}\Big)$$
is a rational $\overline{\mathcal{P}}$- inner function. 
By the definition \eqref{a_in} of $a_{in}$, the zeros of $a_{in}$ are $\beta_1, \dots, \beta_m$, while, since $A$ is an outer polynomial, $A$ has no zeros in $\Bbb{D}$. Hence the zeros of $a = a_{in}\dfrac{A}{D}$ in $\Bbb{D}$ are $\beta_1, \dots, \beta_m$, as required for (1).
\end{proof}

\begin{thm}\label{producepentainnerfunc}
Let $x=(a, s, p)$ be a rational $\overline{\mathcal{P}}$-inner function of degree $(m+n, n)$ such that
\begin{enumerate}
\item the zeros of $a$, repeated according to multiplicity, are $\beta_1, \beta_2, \dots, \beta_m \in \Bbb{D}$,
\item the zeros of $s$, repeated according to multiplicity, are $\alpha_1, \alpha_2, \dots, \alpha_{k_0} \in \Bbb{D}$ and \\ $\eta_1, \eta_2, \dots, \eta_{k_1} \in \Bbb{T}$, where $2k_0+k_1=n$,
\item the royal nodes of $(s, p)$ are $\sigma_1, \dots, \sigma_n \in \overline{\Bbb{D}}$.
\end{enumerate}
There exists some choice of $c \in \Bbb{T}, \; t_{+} \; \textgreater \; 0, \; t \in \Bbb{R} \setminus \{0\}$ and $\omega \in \Bbb{T}$ such that the recipe in Theorem \ref{constructpentainnfunc} with these choices produces the function $x$.
\end{thm}

\begin{proof}
By Lemma \ref{pentagammainner}, $h=(s, p)$ is a rational $\Gamma$-inner function of degree $n$. As in \cite[Proposition 4.9]{ALY18}, there exists some choice of $t_{+} \; \textgreater \; 0, \; t \in \Bbb{R} \setminus \{0\}$ and $\omega \in \Bbb{T}$ such that the recipe of \cite[Theorem 4.8]{ALY18} produces the function $h$. Let us give those steps. \\
By \cite[Proposition 2.2]{ALY18}, there exist polynomials $E_1$ and $D_1$ such that deg$(E_1)$, deg$(D_1) \leq n$, $E_1$ is $n$-symmetric, $D_1(\lambda) \neq 0$ on $\overline{\Bbb{D}}$, and
$$s = \dfrac{E_1}{D_1} \; \text{ and } \; \; p = \dfrac{D_{1}^{\sim n}}{D_1} \text{ on } \overline{\Bbb{D}}.$$
By hypothesis, the zeros of $s$, repeated according to multiplicity, are $\alpha_1, \alpha_2, \dots, \alpha_{k_0}$ and $\eta_1, \eta_2, \dots, \eta_{k_1}$, where $2k_0+k_1=n$. Since $E_1$ is $n$-symmetric, by \cite[Lemma 4.6]{ALY18}, there exists $t \in \Bbb{R} \setminus \{0\}$ such that
$$E_1(\lambda) = t\prod_{j=1}^{k_0} (\lambda-\alpha_j)(1-\overline{\alpha_j}\lambda) \prod_{j=1}^{k_1} i\textup{e}^{-i \theta_j/2}(\lambda-\eta_j),$$
where $\eta_j = e^{i \theta_j}$ for $j =1, \dots, k_1.$
The royal nodes of $h$ are assumed to be $\sigma_1, \dots, \sigma_n$. By \cite[Proposition 4.5]{ALY18}, for the royal polynomial $R_1$ of $h$, there exists $t_{+} \; \textgreater \; 0$ such that
$$R_1(\lambda) = t_{+} \prod_{j=1}^{n} Q_{\sigma_j}(\lambda),$$
where $Q_{\sigma_j}(\lambda)=  (\lambda-\sigma_j)(1-\overline{\sigma_j}\lambda)$, $j= 1, \dots,n$.
Since $E_1$ and $R_1$ coincide with $E$ and $R$ in the construction of Theorem \ref{constructpentainnfunc}, for a suitable choice of $t_{+} \; \textgreater \; 0$ and $t \in \Bbb{R} \setminus \{0\}$, $D_1$ is a permissible choice for $\omega D$ for some $\omega \in \Bbb{T}$, as a solution of the equation \eqref{RED}. \newline
By assumption the zeros of $a$, repeated according to multiplicity, are $\beta_1, \beta_2, \dots, \beta_m \in \Bbb{D}$. Then the inner part of $a$ will be equal to $a_{in}^{1} = c_1 \prod\limits_{i=1}^{m}B_{\beta_i}$ where $|c_1|=1$. For the outer part of $a$ there is an outer polynomial $A_1$ such that
\begin{eqnarray*}
|A_1(\lambda)|^2 & = & |D_1(\lambda)|^2-\frac{1}{4}|E_1(\lambda)|^2 \\
& = & |D(\lambda)|^2-\frac{1}{4}|E(\lambda)|^2 \\
& = & \lambda^{-n}R(\lambda),
\end{eqnarray*}
for $\lambda \in \Bbb{T}$. By equation \eqref{RA}, $A_1 = c_2 A$ up to a constant $c_2$ such that $|c_2|=1$. Also, $a_{in}^{1}$ coincides with $a_{in}$ for a suitable choice of $c \in \Bbb{T}$. Hence the construction of Theorem \ref{constructpentainnfunc} yields $x=(a, s, p)$ for the appropriate choices of $t_{+} \; \textgreater \; 0$, $t \in \Bbb{R} \setminus \{0\}$, $\omega$ and $c \in \Bbb{T}$.
\end{proof}

\section{A special case of Schwarz lemma for $\mathcal{P}$}\label{SecspecaseofSchwarzlemforP}

The classical Schwarz lemma gives a solvability criterion for a two-point interpolation problem in $\D$. In \cite{ALY2015}  a simple  analogue of Schwarz lemma the for two-point $\mu$-synthesis was given.
We consider a general linear subspace $E$ of $\C^{n\times m}$ and the corresponding $\mu_E$ on $\C^{m\times n}$, as in equation \eqref{defmu}.
We shall denote by $N$ the Nevanlinna class of functions on the disc \cite{rudin} and if $F$ is a matricial function on $\D$ then we write $F\in N$ to mean that each entry of $F$ belongs to $N$.  It then follows from Fatou's Theorem that if $F\in N$ is an $m\times n$-matrix-valued function then
\[
\lim_{r\to 1-} F(r\la) \mbox{ exists for  almost all }  \la \in\T.
\]
The following Schwarz lemma was proved in \cite[Proposition 10.3]{ALY2015}.
\begin{prop}\label{schwarz_L_B_mu}
Let $\la_0 \in \D \setminus \{0 \}$, let $W \in \C^{m\times n}$ and let $E$ be a subset of $\C^{n\times m}$. There exists $F \in  N\cap\hol(\D,\C^{m\times n})$
such that 
\begin{enumerate}
\item $F(0) =0$ and $F(\la_0)= W$,
\item $\mu_E(F(\la)) < 1 $ for all $\la \in \D$
\end{enumerate}
if and only if $\mu_E(W) \le |\la_0|$. 
\end{prop}
In this section we consider a simple case of a Schwarz lemma for the pentablock.
We will need the following elementary technical lemma.

\begin{lem}\label{lem171021}
Let $A = $
$\begin{bmatrix} \lambda_1 & 0 \\
a & \lambda_2
\end{bmatrix}$,
where $\lambda_1, \; \lambda_2, \; a \in \Bbb{C}$. Then the following conditions are equivalent:\\
{\rm (i)} $\lambda_1, \lambda_2 \in \overline{\Bbb{D}}, \; |a| \leq (1 - |\lambda_1|^2)^{\frac{1}{2}}(1 - |\lambda_2|^2)^{\frac{1}{2}}$,\\
{\rm (ii)} $\|A\| \leq 1$,\\
{\rm (iii)} $1-A^*A \geq 0$.
\end{lem}

\begin{defn}\label{H^infty(D,C^2*2)}
$H^{\infty}(\Bbb{D}, \Bbb{C}^{2\times 2})$ denotes the space of bounded analytic $2\times 2$ matrix-valued functions on $\Bbb{D}$ with the supremum norm:
$$\|f\|_{H^\infty} = \sup\limits_{z \in \Bbb{D}}\|f(z)\|_{\Bbb{C}^{2\times 2}}.$$
\end{defn}

\begin{defn}\label{L^infty(T,C^2*2)}
$L^{\infty}(\Bbb{T}, \Bbb{C}^{2\times 2})$ denotes the space of essentially bounded Lebesgue-measurable $2\times 2$ matrix-valued functions on $\Bbb{T}$ with the essential supremum norm:
$$\|f\|_{L^\infty} = \textup{ess}\sup\limits_{|z|=1}\|f(z)\|_{\Bbb{C}^{2\times 2}}.$$
\end{defn}

\begin{lem}\label{lem231121}
If $g \in H^{\infty}(\Bbb{D}, \Bbb{C}^{2\times 2})$ and $\lambda_0 \in \Bbb{D}$ then $\|g(\lambda_0)\|_{\Bbb{C}^{2\times 2}} \leq \|g\|_{L^\infty}$.
\end{lem}

\begin{proof}
Consider any unit vectors $x, y \in \Bbb{C}^{2}$ and the scalar function
\begin{eqnarray*}
f & : & \Bbb{D} \rightarrow \Bbb{C} \\
& : & \lambda \longmapsto \langle g(\lambda)x, y \rangle_{\Bbb{C}^2}.
\end{eqnarray*}
Note that, for every $\lambda \in \Bbb{D}$, since $\|x\|_{\Bbb{C}^2} = \|y\|_{\Bbb{C}^2} = 1$
\begin{eqnarray*}
|f(\lambda)| = |\left\langle g(\lambda)x, y \right\rangle_{\Bbb{C}^2}| & \leq & \|g(\lambda)x\|_{\Bbb{C}^2} \; \|y\|_{\Bbb{C}^2} \; \; \; \; \; \text{(Cauchy-Schwarz inequality)} \\
& \leq & \|g(\lambda)\|_{\Bbb{C}^{2\times 2}} \; \|x\|_{\Bbb{C}^2} \; \|y\|_{\Bbb{C}^2} \\
& \leq & \|g\|_{H^\infty}.
\end{eqnarray*}
Thus $f$ is bounded on $\Bbb{D}$. Since $g$ is analytic on $\Bbb{D}$, it is easy to show that 
 $f$ is analytic on $\Bbb{D}$ and, for every $z_0 \in \Bbb{D}, \; f'(z_0) = \langle g'(z_0)x, y \rangle_{\Bbb{C}^2}$.
By the maximum principle for scalar analytic functions,
for every $\lambda_0 \in \Bbb{D}, \; |f(\lambda_0)| \leq \textup{ess} \sup\limits_{z \in \Bbb{T}} |f(z)|$, and so
\begin{eqnarray*}
\left|\left\langle g(\lambda_0)x, y \right\rangle_{\Bbb{C}^2}\right| & \leq & \textup{ess} \sup\limits_{z \in \Bbb{T}} \left|\left\langle g(z)x, y \right\rangle\right| \\
& \leq & \textup{ess} \sup\limits_{z \in \Bbb{T}} \|g(z)\|_{\Bbb{C}^{2\times 2}} = \|g\|_{L^{\infty}}.
\end{eqnarray*}
Take the supremum of both sides in this inequality over unit vectors $x, y$ to get
$$\|g(\lambda_0)\|_{\Bbb{C}^{2\times 2}} \leq \|g\|_{L^{\infty}}.$$
\end{proof}

The following statement is known and  follows easily from Lemma \ref{lem231121}.

\begin{cor}\label{cor231121}
If $F \in H^{\infty}(\Bbb{D}, \Bbb{C}^{2\times 2})$ and $F(0) = 0$ then, for any $\lambda_0 \in \Bbb{D}$,
$$\|F(\lambda_0)\|_{\Bbb{C}^{2\times 2}} \leq |\lambda_0|\;\|F\|_{H^{\infty}}.$$
\end{cor}

We next describe a special case of a Schwarz lemma for the pentablock.

\begin{thm}\label{prop171021}
Let $\lambda_0 \in \Bbb{D} \setminus \{0\}$, and $(a_0, s_0, p_0) \in \overline{\mathcal{P}}$, where $s_0 = \lambda_1+\lambda_2, \; p_0 = \lambda_1\lambda_2$, for some $\lambda_1, \lambda_2 \in \Bbb{D}$. Then the following conditions are equivalent:\\
{\rm (i)} 
 $|\lambda_1| \leq |\lambda_0|, \; |\lambda_2| \leq |\lambda_0|$, and
\begin{equation}\label{eq525}
|a_0| \leq |\lambda_0|\left(1-\left|\dfrac{\lambda_1}{\lambda_0}\right|^2\right)^{\frac{1}{2}}\left(1-\left|\dfrac{\lambda_2}{\lambda_0}\right|^2\right)^{\frac{1}{2}}.\\
\end{equation}
{\rm (ii)}  There exists an analytic map $F : \Bbb{D} \rightarrow \overline{\Bbb{B}^{2\times 2}}$ such that
$$F(0) = 0 \text{ and } F(\lambda_0) = \begin{bmatrix} \lambda_1 & 0 \\ a_0 & \lambda _2 \end{bmatrix}.$$
Furthermore, if {\rm (i)} holds and $x = \pi \circ F$, then $x$ is an analytic map from $\Bbb{D}$ to $\overline{\mathcal{P}}$ such that $x(0) = (0,0,0)$ and $x(\lambda_0) = (a_0, s_0, p_0)$.
\end{thm}

\begin{proof}
(i) $\Rightarrow$ (ii) By assumption, $|\lambda_1| \leq |\lambda_0|, \; |\lambda_2| \leq |\lambda_0|$ and
$$|a_0| \leq |\lambda_0|\left(1-\left|\dfrac{\lambda_1}{\lambda_0}\right|^2\right)^{\frac{1}{2}}\left(1-\left|\dfrac{\lambda_2}{\lambda_0}\right|^2\right)^{\frac{1}{2}}.$$
Define
\begin{equation}\label{F(lambda)**}
F(\lambda) = \frac{\lambda}{\lambda_0}\begin{bmatrix} \lambda_1 & 0 \\ a_0 & \lambda_2 \end{bmatrix} = \lambda \begin{bmatrix} \lambda_1 / \lambda_0 & 0 \\ a_0 / \lambda_0 & \lambda_2 / \lambda_0 \end{bmatrix}.
\end{equation}
By Lemma \ref{lem171021}, 
$$\left\Vert\begin{bmatrix} \dfrac{\lambda_1}{\lambda_0} & 0 \\ \dfrac{a_0}{\lambda_0} & \dfrac{\lambda_2}{\lambda_0} \end{bmatrix}\right\Vert_{\Bbb{C}^{2\times 2}} \leq 1.$$
Hence
$\|F(\lambda)\| \leq |\lambda| $ for all $\lambda \in \Bbb{D}$, and
so $\|F\|_{\infty} \le 1$.  From the definition \eqref{F(lambda)**} of $F$, we have
$$F(0) = \begin{bmatrix} 0 & 0 \\ 0 & 0 \end{bmatrix} \text{ and } F(\lambda_0) = \begin{bmatrix} \lambda_1 & 0 \\ a_0 & \lambda_2 \end{bmatrix}.$$
(ii) $\Rightarrow$ (i) Suppose (ii) is satisfied. By Corollary \ref{cor231121},
$$\left\Vert\begin{bmatrix} \lambda_1 & 0 \\ a_0 & \lambda_2 \end{bmatrix}\right\Vert_{\Bbb{C}^{2\times 2}} = \|F(\lambda_0)\|_{\Bbb{C}^{2\times 2}} \leq |\lambda_0| \; \|F\|_{H^\infty}.$$
By assumption $\|F\|_{H^\infty} \leq 1$, and so
$\left\Vert\begin{bmatrix} \lambda_1 & 0 \\ a_0 & \lambda_2 \end{bmatrix}\right\Vert_{\Bbb{C}^{2\times 2}} \leq |\lambda_0|$, hence,
$\left\Vert\begin{bmatrix} \dfrac{\lambda_1}{\lambda_0} & 0 \\ \dfrac{a_0}{\lambda_0} & \dfrac{\lambda_2}{\lambda_0} \end{bmatrix}\right\Vert_{\Bbb{C}^{2\times 2}} \leq 1. \; $ 
By Lemma \ref{lem171021},
$$\left\Vert\begin{bmatrix} \dfrac{\lambda_1}{\lambda_0} & 0 \\ \dfrac{a_0}{\lambda_0} & \dfrac{\lambda_2}{\lambda_0} \end{bmatrix}\right\Vert_{\Bbb{C}^{2\times 2}} \leq 1 \text{ if and only if }$$
$$|a_0| \leq |\lambda_0|\left(1-\left|\dfrac{\lambda_1}{\lambda_0}\right|^2\right)^{\frac{1}{2}}\left(1-\left|\dfrac{\lambda_2}{\lambda_0}\right|^2\right)^{\frac{1}{2}}, \; \; |\lambda_1| \leq |\lambda_0| \text{ and } |\lambda_2| \leq |\lambda_0|.$$

Let us consider $x = \pi \circ F$ on $\Bbb{D}$.
By assumption $(a_0, s_0, p_0) \in \overline{\mathcal{P}}$. By Theorem \ref{pentathm5.3} (6), since $\|F(\lambda)\| \leq 1$ for each $\lambda \in \Bbb{D}$, $x(\lambda) = \pi(F(\lambda)) = \dfrac{\lambda}{\lambda_0}(a_0, s_0, p_0)$ maps $\Bbb{D}$ to $\overline{\mathcal{P}}$. Therefore $x : \Bbb{D} \rightarrow \overline{\mathcal{P}}$ is analytic on $\Bbb{D}$ and maps 0 to $(0, 0, 0)$ and $\lambda_0$ to $(a_0, s_0, p_0)$.
\end{proof}

\section{A Schwarz Lemma for the symmetrized bidisc $\Gamma$}\label{SecSchwarzlemforGamma}

In \cite{AY2001} Agler and Young proved the following theorems.

\begin{thm}\label{SchwarzlemforGammathm1.1}
\textup{\cite[Theorem 1.1]{AY2001}} Let $\lambda_0 \in \Bbb{D}$ and $(s_0, p_0) \in \Gamma$. The following conditions are equivalent:
\begin{enumerate}
\item There exists an analytic function $\varphi : \Bbb{D} \rightarrow \Gamma$ such that $\varphi(0) = (0, 0)$ and $\varphi(\lambda_0) = (s_0, p_0)$;
\item $|s_0| \; \textless \; 2$ and
$$\dfrac{2|s_0 - p_0\overline{s_0}| + |s_{0}^{2} - 4p_0|}{4 - |s_0|^2} \leq |\lambda_0|;$$
\item 
\begin{equation}\label{lam0s0p0}
\big| |\lambda_0|^2s_0 - p_0\overline{s_0}\big| + |p_0|^2 + (1 - |\lambda_0|^2) \dfrac{|s_0|^2}{4} - |\lambda_0|^2 \leq 0;
\end{equation}
\item $$|s_0| \leq \frac{2}{1 - |\lambda_0|^2}(|\lambda_0| |1 - p_0\overline{\omega}^2| - \big| |\lambda_0|^2 - p_0 \overline{\omega}^2\big|),$$
\noindent where $\omega$ is a complex number of unit modulus such that $s_0 = |s_0|\omega$.
\end{enumerate}
Moreover, for any analytic function $\varphi = (\varphi_1, \varphi_2) : \Bbb{D} \rightarrow \Gamma$ such that $\varphi(0) = (0, 0)$,
$$\frac{1}{2}|\varphi_{1}^{'}(0)| + |\varphi_{2}^{'}(0)| \leq 1.$$
\end{thm}

The following theorem shows the construction of an interpolating function $\varphi$ satisfying the inequalities of Theorem \ref{SchwarzlemforGammathm1.1} with equality.

\begin{thm}\label{SchwarzlemforGammathm1.4}
\textup{\cite[Theorem 1.4]{AY2001}} Let $\lambda_0 \in \Bbb{D}$, and $(s_0, p_0) \in \Gamma$ be such that $\lambda_0 \neq 0$, $|s_0| <2$ and
$$\dfrac{2|s_0 - p_0\overline{s}_0| + |s_{0}^{2} - 4p_0|}{4 - |s_0|^2} = |\lambda_0|.$$
Then there exists an analytic function $\varphi : \Bbb{D} \rightarrow \Gamma$ such that $\varphi(0) = (0, 0)$ and $\varphi(\lambda_0) = (s_0, p_0)$, given explicitly as follows. \\
If $|p_0| = |\lambda_0|$, then 
\begin{equation}\label{phi0}
\varphi(\lambda) = (0, \omega \lambda),
\end{equation}
where $\omega$ is a complex number of unit modulus such that $\omega \lambda_0 = p_0$. \\
If $|p_0| \; \textless \; |\lambda_0|$, then $\varphi = (\varphi_1, \varphi_2)$ where
\begin{equation}\label{phi1}
\varphi_1(\lambda) = \frac{c \zeta \lambda}{(1 - \overline{\lambda}_{0}\lambda) (1 + \overline{p}_1 \zeta^2 v(\lambda))},
\end{equation}
\[v(\lambda) = \frac{\lambda - \lambda_0}{1 - \overline{\lambda}_0 \lambda}, \; \; \; \; \; \zeta \lambda_0|s_0| = |\lambda_0| s_0, \; \; \; \; \; |\zeta| = 1,\]
\[p_1 = \frac{p_0}{\lambda_0}, \; \; \; \; \; c = \frac{2}{|\lambda_0|}\{|\overline{\lambda}_0 - \overline{p}_0 \lambda_0 \zeta^2| - |\lambda_0^2 \zeta^2 - p_0|\},\]
\begin{equation}\label{phi2}
\varphi_2(\lambda) = \frac{\lambda(\zeta^2 v(\lambda) + p_1)}{1 + \overline{p}_1 \zeta^2 v(\lambda)}.
\end{equation}
\end{thm}

\begin{lem}\label{cor1302}
Consider the rational function $\varphi = (\varphi_1, \varphi_2): \Bbb{D} \rightarrow \Gamma$, where $\varphi_1, \varphi_2$ are defined as in equations \eqref{phi1} and \eqref{phi2} above. Define the polynomials $E$ and $D$ by the equations:
$$E(\lambda) = c \lambda,$$
$$D(\lambda) = \overline{\zeta}\{(1-\overline{\lambda_0}\lambda) + \overline{p_1}\zeta^2 (\lambda - \lambda_0)\},$$
where
$$|\zeta| = 1, \; \; \; p_1 = \dfrac{p_0}{\lambda_0}, \; \; \; \; \; c = \dfrac{2}{|\lambda_0|}\{|\overline{\lambda}_0 - \overline{p}_0 \lambda_0 \zeta^2| - |\lambda_0^2 \zeta^2 - p_0|\}.$$
Then $\varphi_1 = \dfrac{E}{D}$ and $\varphi_2 = \dfrac{D^{\sim 2}}{D}$. Moreover, $E^{\sim 2} = E$ and $|E(\lambda)| \leq 2 |D(\lambda)|$ on $\overline{\Bbb{D}}$.
\end{lem}

\begin{proof}
Let us check that $\varphi_1(\lambda) = \dfrac{E(\lambda)}{D(\lambda)}$.
\begin{eqnarray*}
\frac{E(\lambda)}{D(\lambda)} & = & \frac{c\lambda}{\overline{\zeta}\{(1-\overline{\lambda_0}\lambda) + \overline{p_1}\zeta^2 (\lambda - \lambda_0)\}} \times \frac{\zeta}{\zeta} \\
& = & \frac{c \zeta \lambda}{(1-\overline{\lambda_0}\lambda) + \overline{p_1}\zeta^2 (\lambda - \lambda_0)} \\
& = & \frac{c \zeta \lambda}{(1-\overline{\lambda_0}\lambda) (1 + \overline{p_1}\zeta^2 v(\lambda))} = \varphi_1(\lambda).
\end{eqnarray*}
To check that $\varphi_2(\lambda) = \dfrac{D^{\sim 2}(\lambda)}{D(\lambda)}$, we need to find $D^{\sim 2}(\lambda)$.
\begin{eqnarray*}
D^{\sim 2}(\lambda) = \lambda^2 \overline{D(1 / \overline{\lambda})} & = & \lambda^2 \overline{\overline{\zeta}\left\{\left(1- \frac{\overline{\lambda_0}}{\overline{\lambda}}\right) + \overline{p_1}\zeta^2\left(\frac{\ 1  \ }{\ \overline{\lambda}  \ }-\lambda_0\right)\right\}} \\
& = & \lambda^2 \zeta \left\{\left(1- \frac{\lambda_0}{\lambda}\right) + p_1 \overline{\zeta}^2 \left(\frac{1}{\lambda} - \overline{\lambda_0}\right)\right\} \\
& = & \lambda \zeta (\lambda - \lambda_0) + \lambda p_1 \overline{\zeta}(1-\overline{\lambda_0}\lambda).
\end{eqnarray*}
Now,
\begin{eqnarray*}
\frac{D^{\sim 2}(\lambda)}{D(\lambda)} & = & \frac{\lambda \zeta (\lambda - \lambda_0) + \lambda p_1 \overline{\zeta}(1-\overline{\lambda_0}\lambda)}{\overline{\zeta}\{(1-\overline{\lambda_0}\lambda) + \overline{p_1}\zeta^2 (\lambda - \lambda_0)\}} \times \frac{\left(\dfrac{\zeta}{1-\overline{\lambda_0}\lambda}\right)}{\left(\dfrac{\zeta}{1-\overline{\lambda_0}\lambda}\right)} \\
& = & \frac{\lambda \zeta^2 \left(\dfrac{\lambda - \lambda_0}{1-\overline{\lambda_0}\lambda}\right) + \lambda p_1}{1+\overline{p_1}\zeta^2 \left(\dfrac{\lambda - \lambda_0}{1-\overline{\lambda_0}\lambda}\right)} \\
& = & \frac{\lambda (\zeta^2 v(\lambda)+ p_1)}{1+\overline{p_1} \zeta^2 v(\lambda)} = \varphi_2(\lambda),
\end{eqnarray*}
where $v(\lambda) = \dfrac{\lambda - \lambda_0}{1-\overline{\lambda_0}\lambda}$. \\
We would like to show that $E^{\sim 2} = E$. For $\lambda \in \Bbb{D}$,
\begin{eqnarray*}
E^{\sim 2}(\lambda) = \lambda^2 \overline{E(1/\overline{\lambda})} & = & \lambda^2 \overline{\left(c \; \frac{\ 1 \ }{\ \overline{\lambda} \ }\right)} \\
& = & c \; \lambda = E(\lambda), \;\; \text{since} \; c \in \Bbb{R}.
\end{eqnarray*}
By assumption,  $\varphi = (\varphi_1, \varphi_2): \Bbb{D} \rightarrow \Gamma$, and we have proved that 
$\varphi_1 = \dfrac{E}{D}$, thus  $|E(\lambda)| \leq 2 |D(\lambda)|$ on $\overline{\Bbb{D}}$.
\end{proof}

\begin{prop}\label{SchwarzlemforGammathm1.4Gammainner}
Let $h = (s, p)$ be the function from $\Bbb{D}$ to $\Gamma$ defined by
$$s(\lambda) = \varphi_1(\lambda), \; p(\lambda) = \varphi_2(\lambda), \; \lambda \in \Bbb{D},$$
as in equations \eqref{phi1} and \eqref{phi2}. Then $h$ is a rational $\Gamma$-inner function of degree 2.
\end{prop}

\begin{proof}
By Lemma \ref{cor1302} and by the converse part of \cite[Proposition 2.2]{ALY18}, $h$ is a rational $\Gamma$-inner function.\\

Another way to prove that $h$ is a rational $\Gamma$-inner function is as follows.
One can easily see that $h = (s, p)$ is a rational function, and so there are only finitely many  singularities of this function. Hence we can extend  $h$ continuously to almost all points in $\Bbb{T}$.

Let us show that for almost all $\lambda \in \Bbb{T}, \; h(\lambda) \in b\Gamma$. We need to show that, for almost all $\lambda \in \Bbb{T}, \; |p(\lambda)| = 1, \; |s(\lambda)| \leq 2$ and $s(\lambda) = \overline{s(\lambda)} p(\lambda)$. Since $v(\lambda) = \dfrac{\lambda-\lambda_0}{1-\overline{\lambda_0}\lambda}$ is an inner function from $\Bbb{D}$ to $\overline{\Bbb{D}}$, for almost all $\lambda \in \Bbb{T}, \; |v(\lambda)| = 1$.
\begin{eqnarray*}
|p(\lambda)| & = & \frac{|\lambda| |\zeta^2 v(\lambda) + p_1|}{|1 + \overline{p_1} \zeta^2 v(\lambda)|} = \frac{|\zeta^2 v(\lambda) + p_1|}{|\zeta^2 \overline{\zeta^2} (v \overline{v})(\lambda) + \overline{p_1} \zeta^2 v(\lambda)|} \\
& = & \frac{|\zeta^2 v(\lambda) + p_1|}{|\zeta^2 v(\lambda) \overline{(\zeta^2 v(\lambda) + p_1)}|} = \frac{|\zeta^2 v(\lambda) + p_1|}{|\zeta^2 v(\lambda) + p_1|} = 1,
\end{eqnarray*}
for almost all $\lambda \in \Bbb{T}$. \\
In \cite[Theorem 1.5]{AY2001} it was proved that, for all $\lambda \in \Bbb{D}$,
\begin{equation}\label{AY2001(1.8)}
\big| |\lambda|^2 s(\lambda) - \overline{s(\lambda)}p(\lambda)\big| + |p(\lambda)|^2 + (1-|\lambda|^2)\dfrac{|s(\lambda)|^2}{4} - |\lambda|^2 = 0.
\end{equation}
Choose a sequence $(r_n)_{n \geq 1}$ such that $ 0 < r_n < 1$ for each $n$ and $\lim\limits_{n \to \infty} r_n = 1$.
Consider $\mu \in  \Bbb{T}$ and let $\lambda = r_n \mu$ in equation \eqref{AY2001(1.8)}.
On letting $n \to \infty$ we find that, for almost all $\mu \in \Bbb{T}$, 
\begin{equation}\label{AY2001(1.8)mu}
\big| |\mu|^2 s(\mu) - \overline{s(\mu)}p(\mu)\big| + |p(\mu)|^2 + (1-|\mu|^2)\dfrac{|s(\mu)|^2}{4} - |\mu|^2 = 0.
\end{equation}
Note that $|\mu| = 1$ and $|p(\mu)|^2=1$, and so equation \eqref{AY2001(1.8)mu} is equivalent to
$$|s(\mu) - \overline{s(\mu)}p(\mu)| = 0.$$
Hence $s(\mu) = \overline{s(\mu)}p(\mu)$ for almost all $\mu \in \Bbb{T}$. 
Note that for all $\lambda \in \Bbb{D}, \; h(\lambda) = (s(\lambda), p(\lambda)) \in \Gamma$, which means $|s(\lambda)| \leq 2$ for all $\lambda \in \Bbb{D}$. Thus, for almost all $\mu \in \Bbb{T}, \; |s(\mu)| \leq 2$. 
By \cite[Proposition 3.3]{ALY18}, $\text{deg}(h) = \text{deg}(p)$. By Lemma \ref{cor1302},
$$p(\lambda) = \varphi_2(\lambda) = \frac{D^{\sim 2}(\lambda)}{D(\lambda)} = \dfrac{\lambda \zeta (\lambda-\lambda_0) + \lambda p_1 \overline{\zeta} (1-\overline{\lambda_0}\lambda)}{\overline{\zeta}\{(1-\overline{\lambda_0}\lambda) + \overline{p_1}\zeta^2 (\lambda - \lambda_0)\}},$$
where $D(\lambda) = \overline{\zeta}\{(1-\overline{\lambda_0}\lambda) + \overline{p_1}\zeta^2 (\lambda - \lambda_0)\}$.
Since $\text{deg}(D^{\sim 2}) = 2$ and $\text{deg}(D) = 1$, then $\text{deg}(p) = 2$. Therefore $\text{deg}(h) = 2$.
\end{proof}

\section{Sharpness of the Schwarz lemma for $\mathcal{P}$}\label{SecSchwarzlemforP}

Recall  necessary conditions for a Schwarz lemma for $\mathcal{P}$ which  were given in \cite[Proposition 11.1]{ALY2015}.

\begin{prop}\label{pentaprop11.1}
\textup{\cite[Proposition 11.1]{ALY2015}} Let $\lambda_0 \in \Bbb{D}\setminus\{0\}$ and $(a_0, s_0, p_0) \in \overline{\mathcal{P}}$. If $x \in \textup{Hol}(\Bbb{D}, \mathcal{P})$ satisfies $x(0) = (0, 0, 0)$ and $x(\lambda_0) = (a_0, s_0, p_0)$ then $|s_0| <2$,
\begin{equation}\label{11.1}
\frac{2|s_0-\overline{s}_0 p_0|+|s_0^2-4p_0|}{4-|s_0|^2} \leq |\lambda_0|
\end{equation}
and
\begin{equation}\label{11.2}
|a_0|\Big/\left|1-\frac{\frac{1}{2}s_0\overline{\beta}}{1+\sqrt{1-|\beta|^2}}\right| \leq |\lambda_0|,
\end{equation}
where $\beta = (s_0-\overline{s}_0 p_0)/(1-|p_0|^2)$  when $|p_0| < 1$ and 
$\beta = \frac{1}{2}s_0$ when $|p_0| = 1$.
\end{prop}

We prove the following result.

\begin{thm}\label{thmtr2911}
Let $\lambda_0 \in \Bbb{D}\setminus\{0\}$, and $(a_0, s_0, p_0) \in \overline{\mathcal{P}}$ be such that $|s_0| <2$,
\begin{equation}\label{s0p0h0}
\dfrac{2|s_0 - p_0\overline{s}_0| + |s_{0}^{2} - 4p_0|}{4 - |s_0|^2} = |\lambda_0|
\end{equation}
and
\begin{equation}\label{a0s0h0}
|a_0| \; \leq \; |\lambda_0| \sqrt{1-\frac{1}{4}|s_0|^2}.
\end{equation}
Then there exists a rational $\overline{\mathcal{P}}$-inner function $x : \Bbb{D} \rightarrow \overline{\mathcal{P}}$ such that $x(0) = (0, 0, 0)$ and $x(\lambda_0) = (a_0, s_0, p_0)$
given explicitly as follows. \\

\noindent {\rm (i)} If $|p_0| = |\lambda_0|$, then  $s_0= 0$ and
$x(\lambda) = (a(\lambda), 0, \omega \lambda),$
where  $\omega \lambda_0 = p_0, \; \omega \in \Bbb{T}$ and\\

{\rm (a)} if $|a_0| = |\lambda_0|$, then, for $  \lambda \in \Bbb{D}$,  $a (\lambda) = \kappa \lambda, \;$ where $ |\kappa|=1$ and  $\kappa \lambda_0 = a_0$;\\

{\rm (b)} if $|a_0| < |\lambda_0|$, then
$$a(\lambda) = \lambda \dfrac{\lambda - \lambda_0+\eta_0(1-\overline{\lambda_0}\lambda)}{1-\overline{\lambda_0}\lambda+\overline{\eta_0}(\lambda - \lambda_0)}, \; \lambda \in \Bbb{D},$$
and $\eta_0 = \dfrac{a_0}{\lambda_0}$.\\

\noindent {\rm (ii)} If $|p_0| \; \textless \; |\lambda_0|$, then 
$x(\lambda) = (a(\lambda), \varphi_1(\lambda), \varphi_2(\lambda)),$
where $\varphi_1$ is defined by equation \eqref{phi1}, $\varphi_2$ is defined by equation \eqref{phi2} and \\

{\rm (a)} if $|a_0| = |\lambda_0| \sqrt{1-\frac{1}{4}|s_0|^2}$, then, for $  \lambda \in \Bbb{D}$,
$$a(\lambda) =\gamma  \lambda \frac{A(\lambda)}{D(\lambda)},$$
where $ |\gamma|=1$ such that   $\;\gamma \lambda_0 \sqrt{1-\frac{1}{4}|s_0|^2} = a_0$,
$A(\lambda) = b_0(1+b_1\lambda)$
 is an outer   polynomial of degree $1$ such that
\begin{equation}\label{A=ED-thm}
|A(\lambda)|^2 = |D(\lambda)|^2-\frac{1}{4}|E(\lambda)|^2,
\end{equation}
and
\begin{equation}\label{EandD-thm}
E(\lambda) = c \lambda,\; \; 
D(\lambda) = \overline{\zeta}\{(1-\overline{\lambda_0}\lambda) + \overline{p_1}\zeta^2 (\lambda - \lambda_0)\},
\end{equation}
with
$$|\zeta| = 1, \; \; \; p_1 = \dfrac{p_0}{\lambda_0}, \; \; \; \; \; c = \dfrac{2}{|\lambda_0|}\{|\overline{\lambda}_0 - \overline{p}_0 \lambda_0 \zeta^2| - |\lambda_0^2 \zeta^2 - p_0|\}.$$

{\rm (b)} if $|a_0| < |\lambda_0| \sqrt{1-\frac{1}{4}|s_0|^2}$, then, for $  \lambda \in \Bbb{D}$,

$$a(\lambda) = \lambda \left(\frac{\lambda - \lambda_0+\mu_0(1-\overline{\lambda_0}\lambda)}{1-\overline{\lambda_0}\lambda+\overline{\mu_0}(\lambda-\lambda_0)} \right) \frac{A(\lambda)}{D(\lambda)},$$
where $\mu_0 = \dfrac{a_0}{\lambda_0 \sqrt{1-\frac{1}{4}|s_0|^2}}$ and polynomials $A, E, D$ are defined by equations \eqref{EandD-thm} and \eqref{A=ED-thm}.
\end{thm}

\begin{proof}
Since $(a_0, s_0, p_0) \in \overline{\mathcal{P}}$, we have $(s_0, p_0) \in \Gamma$. By assumption the equality \eqref{s0p0h0} holds. Hence, by Theorem \ref{SchwarzlemforGammathm1.4}, there exists a rational analytic function
$$\varphi : \Bbb{D} \rightarrow \Gamma : \lambda \longmapsto (s(\lambda), p(\lambda))$$
such that $\varphi(0) = (0, 0)$ and $\varphi(\lambda_0) = (s_0, p_0)$.
It is easy to see that the function $\varphi = (s, p)$ from $\Bbb{D}$ to $\Gamma$ defined 
as in equation \eqref{phi0} 
 is a rational $\Gamma$-inner function of degree $1$.
 By Proposition \ref{SchwarzlemforGammathm1.4Gammainner}, the function $\varphi = (s, p)$ from $\Bbb{D}$ to $\Gamma$ defined by
$$s(\lambda) = \varphi_1(\lambda), \; p(\lambda) = \varphi_2(\lambda), \; \lambda \in \Bbb{D},$$
as in equations \eqref{phi1} and \eqref{phi2} 
 is a rational $\Gamma$-inner function of degree $2$.
 By Theorem \ref{SchwarzlemforGammathm1.1}, it follows from equation \eqref{lam0s0p0} that $|p_0| \leq |\lambda_0| \; \textless \; 1$. Let us consider two cases. \\

 {\bf Case (i)}. Let  $|p_0| = |\lambda_0|$.  By Theorem \ref{SchwarzlemforGammathm1.4}, if $|p_0|=|\lambda_0|$, then $s_0= 0$ and the function
$\varphi = (0, \omega \lambda)$ from $\Bbb{D}$ to $\Gamma$, where $\omega$ is a complex number of unit modulus such that $\omega \lambda_0 = p_0$, is a rational $\Gamma$-inner function of degree $1$.
Since $s_0=0$, the assumption \eqref{a0s0h0} $|a_0| \leq |\lambda_0| \sqrt{1-\frac{1}{4}|s_0|^2}$ is equivalent to  $|a_0| \leq |\lambda_0|$.
 It is easy to see that, for $\lambda \in \Bbb{D}$, $s(\lambda) = \dfrac{E(\lambda)}{D(\lambda)}$ and $p(\lambda) = \dfrac{D^{\sim n}}{D}(\lambda)$, 
 where polynomials  $E(\lambda)=0$ and $D(\lambda) = \overline{\omega_1}$ and $\omega_1^2=\omega$.
By Theorem \ref{prop446}, we can construct a rational $\overline{\mathcal{P}}$-inner function 
$$x = \left(cB \dfrac{A}{D}, 0, \omega \lambda\right), 
\text{ for an arbitrary finite Blaschke product } B \text{ and } |c| = 1,$$
where $A$ is a non-zero constant such that
$$|A(\lambda)|^2 = |D(\lambda)|^2-\frac{1}{4}|E(\lambda)|^2= |D(\lambda)|^2= | \overline{\omega_1}|^2 =1.$$ 
It is sufficient to construct $a : \Bbb{D} \rightarrow \Bbb{C}$ of the form $a = c B$ with $|c|=1$ such that $a(0) = 0$ and $a(\lambda_0) = a_0$. \\

If $|a_0| = |\lambda_0|$, then, for $\lambda \in \Bbb{D}$, we define 
$a(\lambda) = \kappa \lambda  \; \; \text{for} \;\;\lambda \in \Bbb{D}, $
 where $ |\kappa|=1$ and  $\kappa \lambda_0 = a_0$.
It is easy to see that  $a(0) = 0$ and $a(\lambda_0) = a_0$. \\
 
Let $|a_0| \; \textless \; |\lambda_0|$, and let
$\eta_0 \stackrel{\textup{def}}{=} \frac{a_0}{\lambda_0},$
it is clear that $\eta_0 \in \Bbb{D}$. 
Let us  define $a$ by the formula
$$ a(\lambda) = \lambda B_{\eta_0}^{-1} \circ B_{\lambda_0}(\lambda) \; \; \text{for} \;\;\lambda \in \Bbb{D}. $$
Here the Blaschke factors are defined by 
$$B_{\eta_0}^{-1}(z) = \frac{z+\eta_0}{1+\overline{\eta_0}z} \text{ and } B_{\lambda_0}(z) = \frac{z-\lambda_0}{1-\overline{\lambda_0}z} \; \text{for} \; z \in  \Bbb{D}.$$
It is easy to see that  $a(0) = 0$ and $a(\lambda_0) = a_0$. \\

Define a rational $\overline{\mathcal{P}}$-inner function $x : \Bbb{D} \rightarrow \overline{\mathcal{P}}$ by
$x(\lambda) = (a(\lambda), 0, \omega \lambda),$
where  $\omega \lambda_0 = p_0, \; \omega \in \Bbb{T}$ and\\

{\rm (a)} if $|a_0| = |\lambda_0|$, then, for $  \lambda \in \Bbb{D}$,  $a (\lambda) = \kappa \lambda, \;$ where $ |\kappa|=1$ and  $\kappa \lambda_0 = a_0$;\\

{\rm (b)} if $|a_0| < |\lambda_0|$, then
$$a(\lambda) = \lambda \dfrac{\lambda - \lambda_0+\eta_0(1-\overline{\lambda_0}\lambda)}{1-\overline{\lambda_0}\lambda+\overline{\eta_0}(\lambda - \lambda_0)}, \; \lambda \in \Bbb{D},$$
and $\eta_0 = \dfrac{a_0}{\lambda_0}$. This function $x$ satisfies the conditions $x(0) = (0, 0, 0)$ and $x(\lambda_0) = (a_0, s_0, p_0)$. \\

{\bf Case (ii)}. Let $|p_0| \; \textless \; |\lambda_0|$. By Lemma \ref{cor1302}, for  the rational  $\Gamma$-inner function  $\varphi = (s, p)$ from $\Bbb{D}$ to $\Gamma$ defined by
$$s(\lambda) = \varphi_1(\lambda), \; p(\lambda) = \varphi_2(\lambda), \; \lambda \in \Bbb{D},$$
as in equations \eqref{phi1} and \eqref{phi2}, there exist polynomials $E$ and $D$
\begin{equation}\label{EandD}
E(\lambda) = c \lambda,\; \; 
D(\lambda) = \overline{\zeta}\{(1-\overline{\lambda_0}\lambda) + \overline{p_1}\zeta^2 (\lambda - \lambda_0)\},
\end{equation}
where
$$|\zeta| = 1, \; \; \; p_1 = \dfrac{p_0}{\lambda_0}, \; \; \; \; \; c = \dfrac{2}{|\lambda_0|}\{|\overline{\lambda}_0 - \overline{p}_0 \lambda_0 \zeta^2| - |\lambda_0^2 \zeta^2 - p_0|\},$$
such that $s = \dfrac{E}{D}$ and $p= \dfrac{D^{\sim 2}}{D}$. Moreover, $E^{\sim 2} = E$ and $|E(\lambda)| \leq 2 |D(\lambda)|$ on $\overline{\Bbb{D}}$.

Then, by Theorem \ref{prop446}, we can construct a rational $\overline{\mathcal{P}}$-inner function 
$$x = \left(cB \dfrac{A}{D}, \dfrac{E}{D}, \dfrac{D^{\sim n}}{D}\right), 
\text{  for an arbitrary finite Blaschke product } B \text{ and } |c| = 1,$$
where $A= b_0(1+b_1 \lambda)$ is an outer polynomial of degree $1$ such that
\begin{equation}\label{A=DE}
|A(\lambda)|^2 = |D(\lambda)|^2-\frac{1}{4}|E(\lambda)|^2.
\end{equation}
We would like to construct $a : \Bbb{D} \rightarrow \Bbb{C}$ of the form $a = c B \dfrac{A}{D}$ such that $a(0) = 0$ and $a(\lambda_0) = a_0$. \\

Since, for $\lambda \in \Bbb{D}$,
$$\frac{|A(\lambda)|^2}{|D(\lambda)|^2} = 1-\frac{1}{4}|s(\lambda)|^2,$$
we get $\big|\dfrac{A}{D}(\lambda_0) \big| =  \sqrt{1-\frac{1}{4}|s_0|^2}$.
Let $B(\lambda) = \lambda \tilde{B}(\lambda),$  with some finite Blaschke product
 $ \tilde{B}$. Then $B(0) = 0$, and so $a(0) = 0$. Recall we require $a(\lambda_0) = a_0$,
\begin{eqnarray*}
a_0 & = & a(\lambda_0) = c B(\lambda_0) \dfrac{A}{D}(\lambda_0)\\
& = & c \lambda_0 \tilde{B}(\lambda_0) \sqrt{1-\frac{1}{4}|s_0|^2},
\end{eqnarray*}
for some  $|c| = 1$.
By assumption \eqref{a0s0h0},
$$|a_0| \; \leq \; |\lambda_0| \sqrt{1-\frac{1}{4}|s_0|^2} \;\text{and}  \; \; |s_0| <2, \:\; \text{ thus } \frac{|a_0|}{|\lambda_0| \sqrt{1-\frac{1}{4}|s_0|^2}} \; \leq \; 1.$$

Suppose $|a_0| < |\lambda_0| \sqrt{1-\frac{1}{4}|s_0|^2}$, and let
$$\mu_0 \stackrel{\textup{def}}{=} \frac{a_0}{\lambda_0 \sqrt{1-\frac{1}{4}|s_0|^2}}.$$
It is clear that $\mu_0 \in \Bbb{D}$. 
We need to find $\tilde{B} : \Bbb{D} \rightarrow \Bbb{D}$ such that $\tilde{B}(\lambda_0) = \mu_0$. Let $\tilde{B}= B_{\mu_0}^{-1} \circ B_{\lambda_0}$, where
$$B_{\mu_0}^{-1}(z) = \frac{z+\mu_0}{1+\overline{\mu_0}z} \text{ and } B_{\lambda_0}(z) = \frac{z-\lambda_0}{1-\overline{\lambda_0}z}.$$
Then, for all $z \in \Bbb{D}$,
\begin{eqnarray*}
\tilde{B}(z) & = & B_{\mu_0}^{-1} \circ B_{\lambda_0}(z) = \frac{\dfrac{z-\lambda_0}{1-\overline{\lambda_0}z}+\mu_0}{1+\overline{\mu_0}\left(\dfrac{z-\lambda_0}{1-\overline{\lambda_0}z}\right)} \\
& = & \frac{z-\lambda_0+\mu_0(1-\overline{\lambda_0}z)}{1-\overline{\lambda_0}z+\overline{\mu_0}(z-\lambda_0)}.
\end{eqnarray*}
Let us define $a : \Bbb{D} \rightarrow \Bbb{C}$, for $\lambda \in \Bbb{D}$, by
$$a(\lambda) = \lambda \tilde{B}(\lambda) \dfrac{A(\lambda)}{D(\lambda)}= \lambda \left(\frac{\lambda - \lambda_0+\mu_0(1-\overline{\lambda_0}\lambda)}{1-\overline{\lambda_0}\lambda+\overline{\mu_0}(\lambda-\lambda_0)} \right) \frac{A(\lambda)}{D(\lambda)}.$$
Note that $a(0) = 0$ and  $\tilde{B}(\lambda_0) = \dfrac{\mu_0(1-\overline{\lambda_0}\lambda_0)}{1-\overline{\lambda_0}\lambda_0} = \mu_0$. Therefore
\begin{eqnarray*}
a(\lambda_0) & = & \lambda_0 \mu_0 \frac{A(\lambda_0)}{D(\lambda_0)} \\
& = & \lambda_0 \frac{a_0}{\lambda_0 \sqrt{1-\frac{1}{4}|s_0|^2}} \sqrt{1-\frac{1}{4}|s_0|^2} \\
& = & a_0.
\end{eqnarray*}

Hence, in the case when $|p_0| \; \textless \; |\lambda_0|$, we define a 
rational $\overline{\mathcal{P}}$-inner function $x : \Bbb{D} \rightarrow \overline{\mathcal{P}}$ by
$x(\lambda) = (a(\lambda), \varphi_1(\lambda), \varphi_2(\lambda)), \;\ \text{for all } \;\; \lambda \in \Bbb{D},$
where
$\varphi_1$ is defined by equation \eqref{phi1}, $\varphi_2$ is defined by equation \eqref{phi2} and  \\

{\rm (a)} if $|a_0| = |\lambda_0| \sqrt{1-\frac{1}{4}|s_0|^2}$, then, for  $\lambda \in \Bbb{D}$,
$$a(\lambda) = \gamma \lambda \frac{A(\lambda)}{D(\lambda)},$$
where  $ |\gamma|=1$  such that   $\;\gamma \lambda_0 \sqrt{1-\frac{1}{4}|s_0|^2} = a_0$.\\

{\rm (b)} if $|a_0| < |\lambda_0| \sqrt{1-\frac{1}{4}|s_0|^2}$, then

$$a(\lambda) = \lambda \left(\frac{\lambda - \lambda_0+\mu_0(1-\overline{\lambda_0}\lambda)}{1-\overline{\lambda_0}\lambda+\overline{\mu_0}(\lambda-\lambda_0)} \right) \frac{A(\lambda)}{D(\lambda)},$$
where $\mu_0 = \dfrac{a_0}{\lambda_0 \sqrt{1-\frac{1}{4}|s_0|^2}}$, and polynomials $A$, $E$ and $D$ are defined by equations \eqref{EandD} and \eqref{A=DE}.

One can verify that a suitable choice of $A$ is $A(\lambda) = b_0(1+b_1\lambda)$, where
$$|b_0|^2 = |1-\overline{p_1}\zeta^2 \lambda_0|^2 \;\; \text{and} \;\;
|b_1|^2  = 2 \; \frac{|\lambda_0 \zeta^2 - p_1|}{|1-\overline{p_1}\zeta^2 \lambda_0|} - 1.$$

\end{proof}

\begin{thm}\label{Schwarz-P-inner}
Let $\lambda_0 \in \Bbb{D}\setminus\{0\}$ and $(a_0, s_0, p_0) \in \overline{\mathcal{P}}$.
Then the following conditions are equivalent:

{\rm (i)}  there exists a rational $\overline{\mathcal{P}}$-inner function $x=(a, s, p)$,  $x : \Bbb{D} \rightarrow \overline{\mathcal{P}}$ such that $x(0) = (0, 0, 0)$ and $x(\lambda_0) = (a_0, s_0, p_0)$;

{\rm (ii)}  there exists an analytic function $x=(a, s, p)$,  $x : \Bbb{D} \rightarrow \overline{\mathcal{P}}$ such that  $x(0) = (0, 0, 0)$  and $x(\lambda_0) =(a_0, s_0, p_0)$, and $|a_0| \; \leq \; |\lambda_0| \sqrt{1-\frac{1}{4}|s_0|^2}$; 

{\rm (iii)}
\begin{equation}\label{s0p0h0-2}
\dfrac{2|s_0 - p_0\overline{s}_0| + |s_{0}^{2} - 4p_0|}{4 - |s_0|^2} \le |\lambda_0|, \;\;|s_0| <2,
\end{equation}
and
\begin{equation}\label{a0s0h0-2}
|a_0| \; \leq \; |\lambda_0| \sqrt{1-\frac{1}{4}|s_0|^2}.
\end{equation}
\end{thm}

\begin{proof}  We shall prove below that (i) $ \Longleftrightarrow$  (iii), from which it will follow trivially that  (i) $\Rightarrow$ (ii).\\

(ii) $\Rightarrow$ (i)
Suppose (ii) holds, that is,  there exists an analytic function $x_1=(a', s', p')$,  $x_1 : \Bbb{D} \rightarrow \overline{\mathcal{P}}$ such that $x_1(0) = (0, 0, 0)$ and $x_1(\lambda_0) = (a_0, s_0, p_0)$. By Lemma \ref{pentagammainner}, $h_1=(s',p'): \Bbb{D} \rightarrow \Gamma$ is an analytic function such that $h_1(0) = (0, 0)$ and $h_1(\lambda_0) = (s_0, p_0)$. 

By \cite[Theorem 4]{costara2005} (see also \cite[Theorem 8.1]{ALY13-2}), there exists a rational $\Gamma$-inner function $h: \Bbb{D} \rightarrow \Gamma$ satisfying $h(0) = (0, 0)$ and $h(\lambda_0) = (s_0, p_0)$.
Let $E$ and $D$ be polynomials as in equations \eqref{E-D-s-p} (see  \cite[Proposition 2.2]{ALY18}) with
$h= \left(\dfrac{E}{D}, \dfrac{D^{\sim n}}{D}\right)$ on $\Bbb{D}$, where $n = \deg h$.
By Theorem \ref{prop446}, we can construct a rational $\overline{\mathcal{P}}$-inner function 
$$x = \left(a,\dfrac{E}{D}, \dfrac{D^{\sim n}}{D}\right)  =\left(cB \dfrac{A}{D},\dfrac{E}{D}, \dfrac{D^{\sim n}}{D}\right), $$
for any finite Blaschke product $B$ and $c \in \C$ with $ |c| = 1,$
where $A$ is an outer polynomial  such that
$$|A(\lambda)|^2 = |D(\lambda)|^2-\frac{1}{4}|E(\lambda)|^2, \;\; \text{for} \;\;\lambda \in \Bbb{T}.$$ 
Hence, for $\lambda \in \Bbb{D}$,
$$\frac{|A(\lambda)|^2}{|D(\lambda)|^2} = 1-\frac{1}{4}|s(\lambda)|^2.$$
Thus we get $\big|\dfrac{A}{D}(\lambda_0) \big| =  \sqrt{1-\frac{1}{4}|s_0|^2}$.
By assumption, $|a_0| \le |\lambda_0| \sqrt{1-\frac{1}{4}|s_0|^2}$. As in Theorem \ref{thmtr2911}, to satisfy conditions
$x(0) = (0, 0, 0)$ and $x(\lambda_0) = (a_0, s_0, p_0)$,
we define a function $a$ the following way:

{\rm (a)} if $|a_0| = |\lambda_0| \sqrt{1-\frac{1}{4}|s_0|^2}$, then, for  $\lambda \in \Bbb{D}$, 
$$a(\lambda) = \gamma \lambda \frac{A(\lambda)}{D(\lambda)},$$
where  $ |\gamma|=1$ is such that   $\;\gamma \lambda_0 \sqrt{1-\frac{1}{4}|s_0|^2} = a_0$.\\

{\rm (b)} if $|a_0| < |\lambda_0| \sqrt{1-\frac{1}{4}|s_0|^2}$, then
$$a(\lambda) = \lambda \left(\frac{\lambda - \lambda_0+\mu_0(1-\overline{\lambda_0}\lambda)}{1-\overline{\lambda_0}\lambda+\overline{\mu_0}(\lambda-\lambda_0)} \right) \frac{A(\lambda)}{D(\lambda)},$$
where $\mu_0 = \dfrac{a_0}{\lambda_0 \sqrt{1-\frac{1}{4}|s_0|^2}}$.\\
Therefore condition (i) holds.

(iii) $\Rightarrow$ (i) 
By Theorem \ref{SchwarzlemforGammathm1.1}, since condition (iii) holds, there exists an analytic function  $h_1: \Bbb{D} \rightarrow \Gamma$ such that  $h_1(0) = (0, 0)$ and $h_1(\lambda_0) = ( s_0, p_0)$.
By \cite[Theorem 4]{costara2005} (see also \cite[Theorem 8.1]{ALY13-2}), there exists a rational $\Gamma$-inner function $h=(s, p): \Bbb{D} \rightarrow \Gamma$ satisfying $h(0) = (0, 0)$ and $h(\lambda_0) = (s_0, p_0)$.
Let $E$ and $D$ be polynomials as in equations \eqref{E-D-s-p} (see  \cite[Proposition 2.2]{ALY18}) with
$h= \left(\dfrac{E}{D}, \dfrac{D^{\sim n}}{D}\right)$ on $\Bbb{D}$, where $n = \deg h$.
By Theorem \ref{prop446}, we can construct a rational $\overline{\mathcal{P}}$-inner function 
$$x = \left(a,\dfrac{E}{D}, \dfrac{D^{\sim n}}{D}\right)  =\left(cB \dfrac{A}{D},\dfrac{E}{D}, \dfrac{D^{\sim n}}{D}\right), $$
for any finite Blaschke product $B$ and $c \in \C$ with $ |c| = 1,$
where $A$ is an outer polynomial  such that
$$|A(\lambda)|^2 = |D(\lambda)|^2-\frac{1}{4}|E(\lambda)|^2.$$ 
To satisfy the conditions
$x(0) = (0, 0, 0)$ and $x(\lambda_0) = (a_0, s_0, p_0)$,
we define a function $a$ as in Part  (ii) $\Rightarrow$ (i).

Note that in the case when
\[
\dfrac{2|s_0 - p_0\overline{s}_0| + |s_{0}^{2} - 4p_0|}{4 - |s_0|^2} = |\lambda_0|,\;\;|s_0| <2,
\]
and 
\[
|a_0| \; \leq \; |\lambda_0| \sqrt{1-\frac{1}{4}|s_0|^2}
\]
Theorem \ref{thmtr2911} gives the construction of an interpolating  rational $\overline{\mathcal{P}}$-inner function  $x=(a, s, p)$,  $x : \Bbb{D} \rightarrow \overline{\mathcal{P}}$ such that $x(0) = (0, 0, 0)$ and $x(\lambda_0) = (a_0, s_0, p_0)$.\\

(i) $\Rightarrow$ (iii)
Suppose (i) holds, that is,  there exists a rational $\overline{\mathcal{P}}$-inner function $x=(a, s, p)$, 
 $x : \Bbb{D} \rightarrow \overline{\mathcal{P}}$ such that $x(0) = (0, 0, 0)$ and $x(\lambda_0) = (a_0, s_0, p_0)$. Let $\deg x = (m, n)$.

 By Lemma \ref{pentagammainner}, $h=(s,p)$ is a rational $\Gamma$-inner function of degree $n$. Note that $h(0) = (0,0)$ and 
 $h(\lambda_0) = (s_0, p_0)$. By Theorem \ref{SchwarzlemforGammathm1.1}, 
 $|s_0| \; \textless \; 2$ and
$$\dfrac{2|s_0 - p_0\overline{s_0}| + |s_{0}^{2} - 4p_0|}{4 - |s_0|^2} \leq |\lambda_0|.$$

 By Theorem \ref{descripration},  there exist polynomials $A, E, D$ such that
$E^{\sim n} = E$, $D(\lambda) \neq 0$ on $\overline{\Bbb{D}}$, $A$ is an outer polynomial such that $|A(\lambda)|^2 = |D(\lambda)|^2 - \frac{1}{4}|E(\lambda)|^2$ on $\Bbb{T}$, $|E(\lambda)| \leq 2|D(\lambda)|$ on $\overline{\Bbb{D}}$ and 
  $$x=\left(c B \dfrac{A}{D}, \dfrac{E}{D}, \dfrac{D^{\sim n}}{D}\right) \;\; \text{on} \;\;\overline{\Bbb{D}}$$
 for some finite Blaschke product $B$ and $|c|=1$. The function
 $$\lambda \mapsto  a(\lambda) = c B (\lambda)\dfrac{A}{D}(\lambda) $$
 is an analytic map from $\Bbb{D}$ to $\overline{\Bbb{D}}$ such that $a(0) = 0$ and $a(\lambda_0) = a_0$.
 Note that  $A$ and $D$ are outer polynomials on $\overline{\Bbb{D}}$,  and so
 $$ f(\lambda) =\frac{a(\lambda)}{\left(\dfrac{A}{D}(\lambda) \right)}= c B (\lambda)$$
 is an analytic map from $\Bbb{D}$ to $\overline{\Bbb{D}}$ such that $f(0) = 0$. 
 By the classical Schwarz lemma we have
 $$ |f(\lambda)| =  \left| \frac{a(\lambda)}{\dfrac{A}{D}(\lambda)} \right| \leq | \lambda| \;\; \text{for}\;\;\lambda \in \Bbb{D}.$$
 Since 
 $$ \big|\dfrac{A}{D}(\lambda) \big|^2= 1- |s (\lambda)|^2\;\; \text{for} \;\; \lambda \in \Bbb{D},$$
 $$ |f(\lambda)| =  \left|\frac{a(\lambda)}{\sqrt{1-\frac{1}{4}|s (\lambda)|^2}} \right|
 \;\; \text{for} \;\; \lambda \in \Bbb{D}.$$
 Thus
 $$ |a_0| = |a(\lambda_0)| =|f(\lambda_0)| \sqrt{1-\frac{1}{4}|s (\lambda_0)|^2} 
 \leq | \lambda_0| \sqrt{1-\frac{1}{4}|s_0|^2}.$$
 Therefore condition (iii) holds.
\end{proof}

\vspace*{0.5cm}

{\bf Acknowledgments}. The first author  was supported to do PhD study in Newcastle University  by the Government of Saudi Arabia and by King Khalid University, Saudi Arabia.
The second author was partially supported by the UK Engineering and Physical Sciences Research Council grants  EP/N03242X/1.\\

{\bf Statements and Declarations}. 
This paper is based on the first-named author's Ph.D. thesis. Nujood Alshehri has presented her results online to several workshops, the first presentation was to the Young Researchers in Mathematics conference in June 2021, see 
https://sites.google.com/view/yrm-2021/schedule for the abstracts. Results similar to our Theorem \ref{descripration} were  announced on ArXiv in \cite{AJPK} in March, 2022.\\

{\bf Conflict of Interest}.
We are not aware of any conflict of interest with any party.\\

\bibliography{references}

Nujood M. Alshehri,
School of Mathematics, Statistics and Physics, Newcastle University, Newcastle upon Tyne
 NE\textup{1} \textup{7}RU, U.K.;\\
 e-mail: N.M.J.Alshehri1@newcastle.ac.uk
\\

Zinaida A. Lykova,
School of Mathematics, Statistics and Physics, Newcastle University, Newcastle upon Tyne
 NE\textup{1} \textup{7}RU, U.K.;\\
 e-mail: Zinaida.Lykova@newcastle.ac.uk\\

\end{document}